\documentclass[a4paper,12pt]{article}
\usepackage{graphicx,amssymb,amstext,amsmath}

\usepackage[latin1]{inputenc}
\usepackage{amsthm}
\usepackage{caption}
\usepackage{subcaption}
\usepackage{dsfont}
\usepackage{amsmath}
\usepackage{amsfonts}
\usepackage{amssymb}
\usepackage{mathrsfs}
\usepackage{graphicx}
\usepackage{graphicx,color}
\usepackage[all]{xy}
\usepackage[total={5.2in,9.1in},
top=1.4in, left=1.55in, includefoot]{geometry}

\newtheorem{Defi}{Definition}
\newtheorem{R}{Remark}

\newtheorem{T}{Theorem}[section]
\newtheorem{T*}{Theorem}
\newtheorem{Pro}{Proposition}
\newtheorem{C}{Corollary}
\newtheorem{Le}{Lemma}[section]
\newtheorem{Le*}{Lemma}

\newcommand{\ds}{\displaystyle}
\newcommand{\B}{\mathbb{B}}
\newcommand{\A}{\mathbb{A}}
\newcommand{\re}{\mathbb{R}}
\newcommand{\cR}{\mathcal{R}}
\newcommand{\und}{\underline}
\newcommand{\ua}{\underline{a}}

\title{\bf{On the Lagrange and Markov
Dynamical Spectra for Anosov Flows in dimension $3$}}
\author{Sergio Augusto Roma\~na Ibarra}

\date{}
\begin{document}

\maketitle
\begin{abstract}
\noindent We consider the Lagrange and the Markov dynamical spectra associated with a conservative Anosov flow on a compact manifold  of dimension $3$ (including geodesic flows of negative curvature and suspension flows). We show that for a large set of real functions and  typical conservative Anosov flows, both the Lagrange and  Markov dynamical spectra have a non-empty interior.
\end{abstract}

\section{Introduction}

\,\indent The Lagrange and Markov spectra are born in the theory of numbers which  have a dynamic interpretation that can be explored in a more general context using the hyperbolicity of some systems.\\
\indent Many recent results, such as \cite{RM}, {\cite{CMM}}, and \cite{CMR}, show that the Lagrange and Markov spectra are well behaved when looking at the hyperbolic world, including geodesic flows of negative curvature (cf. \cite{PP1}), Teichm\" uller flows (cf. \cite{HMU}), Veech surfaces (cf. \cite{AMU}), among others.\\
Good references for an introduction to these spectra can be found at \cite{CF} and {\cite{OurBook}}.

\subsection{Dynamical Markov and Lagrange Spectra}

Let $M$ be a smooth manifold, $T=\mathbb{Z}$ or $\mathbb{R}$, and $\phi=(\phi^t)_{t\in T}$ be a discrete-time ($T=\mathbb{Z}$) or continuous-time ($T=\mathbb{R}$) smooth dynamical system on $M$, that is, $\phi^t:M\to M$ are smooth diffeomorphisms, $\phi^0=\textrm{id}$, and $\phi^t\circ\phi^s=\phi^{t+s}$ for all $t,s\in \re$.   

Given a compact invariant subset $\Lambda\subset M$ and a function $f:M\to\mathbb{R}$, we denote the \emph{dynamical Markov} and \emph{the Lagrange spectrum}, as $M(\phi, \Lambda, f)$ and $L(\phi, \Lambda, f)$, respectively. They are defined as follows
$$M({\phi, \Lambda, f})=\{m_{\phi, f}(x): x\in\Lambda\} \quad \textrm{and} \quad L({\phi, \Lambda, f})=\{\ell_{\phi, f}(x): x\in\Lambda\},$$
where  
$$m_{\phi, f}(x):=\sup\limits_{t\in \re} f(\phi^t(x)) \quad \textrm{and} \quad \ell_{\phi, f}(x):=\limsup\limits_{t\to+\infty} f(\phi^t(x)).$$
It is easy to see that $L({\phi, \Lambda, f})\subset M({\phi, \Lambda, f})$ (cf. \cite{RM}).\\
\noindent When $\Lambda$ is the whole manifold, we denote $M(\phi, M, f):=M(\phi,f)$ and $L(\phi, M, f):=L(\phi, f)$.\break

\ \\
The first result in the context of discrete dynamic  is due to C. Moreira and S. Roma\~na \cite{RM}, where they proved that: \\
\ \\
\noindent {\bf Theorem }\cite[Main Theorem]{RM} \emph{Let $\Lambda$ be a horseshoe associated to a $C^2$-diffeomorphism $\varphi$ of a surface $N$ such that  $HD(\Lambda)>1$. Then there is arbitrarily close to $\varphi$, a diffeomorphism $\varphi_{0}$ and a $C^{2}$-neighborhood $\mathcal{W}$ of $\varphi_{0}$ such that, if $\Lambda_{\psi}$ denotes the continuation of $\Lambda$ associated to $\psi\in \mathcal{W}$, there is an open and dense set $H_{\psi}\subset C^{1}(N,\re)$ such that for all $f\in H_{\psi}$, we have 
\begin{equation*}
\emph{int} \,L(\psi,\Lambda_{\psi}, f)\neq\emptyset \ \text{and} \ \emph{int}\, M(\psi,\Lambda_{\psi}, f)\neq\emptyset,
\end{equation*}
where $\emph{int}\,A$ denotes the interior of $A$.}\\
\\
This theorem  will be useful to prove our results. 

\subsection{Main Results}\label{main-results}
\indent In this paper, we consider a conservative Anosov flow in dimension $3$, which includes the case of geodesic flow of surface of negative curvature and suspension Anosov flow. More specifically, we consider a three-dimensional connected  $C^{\infty}$-Riemannian manifold $M$ endowed with a finite volume-form. Let $m$ be the measure associated with this form of volume, which we call the \textit{Lebesgue measure}. \\
Let $\mathfrak{X}_{w}^{r}(M)$ be the space of $C^{r}$-conservative vector fields on $M$. Then, we prove the following theorem: 
\begin{T}\label{Theorem 1}
Let $\phi\in \mathfrak{X}_{w}^{r}(M)$, $r \geq {2}$ such that  ${\phi}^{t}$ has a basic set $\Lambda$ with Hausdorff dimension bigger than $2$, then $	C^r$-arbitrarily close to $\phi$  there is an open set $\mathcal{W}\subset \mathfrak{X}_{w}^{r}(M)$ such that for any $X\in \mathcal{W}$ one can find a dense and $C^r$-open subset $\mathcal{U}_{X,\Lambda}\subset C^r(M,\re)$, so that   
$$\emph{int}\,  M(f,X)\neq \emptyset \, \ \text{and}\, \, \ \emph{int}\, L(f,X)\neq \emptyset,$$
whenever $f\in \mathcal{U}_{X,\Lambda}$. Moreover, the above statement holds persistently: for any $Y\in \mathcal{W}$, it holds for any $(f,X)$ in a suitable neighborhood of $\mathcal{U}_{Y,\Lambda}\times \{Y\}$ in $C^{r}({M,\re})\times \mathfrak{X}_{w}^{r}(M)$. 
\end{T}
It worth noting that  the above theorem is valid for transitive Anosov flow which is not necessarily conservative. In this case, the proof is similar to the proof of Theorem \ref{Theorem 1} and does not need the conservative family of perturbations of Subsection \ref{Realiz of the Pert}, for this reason, we will only do the proof for  conservative flows.
\ \\
\ \\
In the case of geodesic flow, let $N$ be a complete surface, let $g_0$ be a smooth ($C^r$, $r\geq 2$) pinched negatively curved Riemannian metric on $N$ (the curvature is bounded above and below by two negative constants). Let  $\phi^t_{0}$ be the geodesic flow on the unit bundle $SN$ and $\phi_0$ the derivative of the geodesic flow $\phi^t_{0}$. In this case, it is well known that $\phi^t_{0}$ is an Anosov flow (cf. \cite{A} and {\cite{Knieper}}). Moreover, if $N$ has finite volume, then $\phi_0\in \mathfrak{X}_{w}^{r}(SN)$, since the Liouville measure is invariant by the geodesic flow (cf. \cite{P}).  

In these conditions, we have the second results 

\begin{C}\label{C1-Theorem 1}Arbitrarily close to $\phi_0$ there exist an open set $\mathcal{V}\subset \mathfrak{X}^{2}_{\omega}(SN)$ such that for any $X\in \mathcal{V}$ one can find a dense and $C^2$-open subset $\mathcal{U}_{X,\Lambda}\subset C^2(SN,\re)$, so that   
$$\emph{int}\, M(f,X)\neq \emptyset \ \ \text{and} \ \ \emph{int}\, L(f,X)\neq \emptyset$$
whenever $f\in \mathcal{U}_{X,\Lambda}$.
Moreover, the above statement holds persistently: for any $Y\in \mathcal{V}$, it holds for any $(f,X)$ in a suitable neighborhood of $\mathcal{U}_{Y,\Lambda}\times \{Y\}$ in $C^{2}({SN,\re})\times \mathfrak{X}^{1}_{w}(SN)$.
\end{C}

\begin{R}
It is important to note that the neighborhood $\mathcal {V}$ of the above corollary is not necessarily a neighborhood of space of  vector field  coming from geodesic flows or Riemannian metric, since small perturbations on the metrics do not produce small perturbation on the geodesic flows.
\end{R} 
Another  interesting class of Anosov flows is the {suspension Anosov flows}, which are the suspension of Anosov diffeomorphisms and is defined by
 a \textit{suspension} flow is defined as follows:
Let $\varphi\colon N \to N$ be an Anosov diffeomorphism of a compact manifold
$N$ and  consider the manifold 
 $$N_{\varphi}=\{(x,r): x\in N \, \, \text{and} \, \, 0\leq r \leq 1 \}/(x,1)\sim (\varphi(x),0).$$
The \textit{Anosov suspension flow} of $\varphi$ is  the flow $\psi^{t}_{_{\varphi}}\colon N_{\varphi} \to N_{\varphi}$ induced by the translated time \break $\psi^{t}\colon N\times\re \to N\times \re$, $\psi^{t}(x,s) = (x, s+t)$. We denoted by $\psi_{\varphi}$ the  derivative vector field of $\psi^{t}_{_{\varphi}}$  (cf. \cite{K} for more details).\\


For this class of Anosov flows we prove:

\begin{C}\label{C2-Theorem 1} Let  ${\psi}^{t}_{\varphi_{_{0}}}$ be is an  Anosov flow which is a suspension  of a $C^2$- Anosov diffeomorphism $\varphi_0$ of a compact surface  $N$. Then, arbitrarily close to $\varphi_0$ there is an open set $\mathcal{W}$ of $C^2$ Anosov diffeomorphisms such that for any $\varphi\in \mathcal{W}$ we have 
$$\emph{int}\, M(\psi_{_{\varphi}},f)\neq \emptyset \ \ \text{and} \ \ \emph{int}\, L(\psi_{_{\varphi}}, f)\neq \emptyset$$
for any $f$ in a dense and $C^{2}$-open subset $\mathcal{U}_{\varphi}$ of $C^{2}(N_{\varphi},\re)$, where $\psi_{\varphi}^{t}$ is the suspension flow associated to $\varphi\in \mathcal{W}$.
\end{C}
\begin{R}\label{R1C2}
If $\varphi_0$, in  \emph{Corollary \ref{C2-Theorem 1}},  is a  $C^2$- conservative Anosov  diffeomorphism, then $\mathcal{W}$ can be considered  contained in the $C^2$ conservative world (see the end of \emph{section \ref{PC2}}).
\end{R}
To prove Theorem \ref{Theorem 1}, we will use the Main Theorem at \cite{RM}, but we point out that its proof is not an immediate consequence of Main Theorem at  \cite{RM}. We comment on three challenges that need to be overcome to apply the Main Theorem at \cite{RM}.\\
The first challenge to overcome  is to show a separation Lemma (see Lemma \ref{L7}) using only the $C^{0}$ stable and unstable foliations of the flow, which allows us to reduce the problem by one dimension. More specifically, the proof of Lemma \ref{L7} involves some techniques of  saturation of surface  by one-dimensional foliations. We first construct a finite number of $C^0$-sections ``transverse" to the Anosov flow  which is saturated by the stable foliation of the basic set $\Lambda$, such that the union of the  box flow neighborhood  of these sections forms a finite cover of $\Lambda$. Then manipulating the hyperbolicity of $\Lambda$, we will make small surgeries to separate this finite number of $C^0$-sections. Finally, we approximate those  $C^0$-cross sections by a $C^\infty$-cross section but now separated. The second challenge is to  produce small conservative perturbations of the flow such that we can obtain, in some way, the conditions to apply the main theorem in \cite{RM} (see section \ref{Realiz of the Pert}). The third  and the hardest challenge to overcome is to construct the set $\mathcal{U}_{X, \Lambda}$ of the statement of Theorem \ref{Theorem 1}, which we need some non-trivial combinatorial arguments  for horseshoe (see Lemma \ref{LIC}).\\ 
\ \\
\textbf{Structure of Paper:} The paper is organized the following way: In Section 2, we give a little introduction of Anosov flows, in Section 3, we will get the tool to reduce the Theorem \ref{Theorem 1} to a problem of dimension two and we will construct the ingredients to define the set $\mathcal{U}_{X, \Lambda}$, in  Section \ref{LMS} we will prove the Theorem \ref{Theorem 2}, which is a bi-dimensional version of theorem \ref{Theorem 1}, with featured for the subsections \ref{Family of Pert} and \ref{CA}, finally, in  Section 5 we will prove the Corollary 1 and Corollary 2.
\section{Preliminaries}


\noindent 
Let $M$ be a complete Riemannian manifold and $\phi^t:M \rightarrow M$ a flow on $M$. We say that a compact invariant set $\Lambda\subset M$ is hyperbolic for $\phi^{t}$ if: there exists a  splitting $T_{\Lambda}M=E^{s}\oplus \phi\oplus E^{u}$ such that  for each $\theta\in \Lambda$
\begin{eqnarray*}
	d\phi^t_{\theta} (E^s(\theta)) &=& E^s(\phi^t(\theta)),\\
	d\phi^t_{\theta} (E^u(\theta)) &=& E^u(\phi^t(\theta)),\\
	||D\phi^t_{\theta}\big{|}_{E^s}|| &\leq& C \lambda^{t},\\
	||D\phi^{-t}_{\theta}\big{|}_{E^u}|| &\leq& C \lambda^{t},\\
\end{eqnarray*}
for all $t\geq 0$ with $ C > 0$ and $0 < \lambda <1$, where  $\phi$ is the vector field derivative of the geodesic vector flow.\\
When $\Lambda=M$ we said that the flow is an \emph{Anosov flow}.\\
\noindent The subbundles $E^s$ and $E^u$ are known to be uniquely integrable. 
From the Stable and Unstable Manifold Theorem \cite{K} it follows that there is $\epsilon>0$ such that for every $x\in \Lambda$ the set 
$$W^{s}_{\epsilon}(x)=\{y: d(\varphi^{t}(x),\varphi^{t}(y))\leq \epsilon \ \text{and} \ d(\varphi^{t}(x),\varphi^{t}(y))\underset{t \to +\infty}{\longrightarrow} 0\}$$
and 
$$W^{u}_{\epsilon}(x)=\{y: d(\varphi^{t}(x),\varphi^{t}(y))\leq \epsilon \ \text{and} \ d(\varphi^{t}(x),\varphi^{t}(y))\underset{t \to -\infty}{\longrightarrow} 0\}$$
are invariant $C^{r}$-manifolds tangent to $E^{s}_{x}$ and $E^{u}_{x}$, respectively, at $x$, where $d$ is the distance on $M$ induced by the Riemannian metric. Then, we call
$W^{s}_{\epsilon}(x)$ the local \emph{strong-stable manifold}  and $W^{u}_{\epsilon}(x)$ the local \emph{strong-unstable manifold}, by abuse of notation, we denote these local manifolds simply writing  $W^{s}_{loc}(x)$ and $W^{u}_{loc}(x)$, respectively.   Moreover, the manifolds $W^{s}_{\epsilon}(x)$ and $W^{u}_{\epsilon}(x)$ vary continuously with $x$ (in general, it is the best one can expect for hyperbolic sets).
Also, if $x\in \Lambda$ one has that 
$$W^{s}(x)=\bigcup_{t\geq 0}\varphi^{-t}(W^{s}_{\epsilon}(\varphi^{t}(x))) \ \ \text{and} \ \ W^{u}(x)=\bigcup_{t\leq 0}\varphi^{-t}W^{u}_{\epsilon}(\varphi^{t}(x))$$
are $C^r$-invariant manifolds immerse in $M$, called the \emph{strong-stable manifold} and \emph{strong-unstable manifold} of $x$, respectively. Finally, the sets 
$$W^{cs}(x)=\bigcup_{t\in\re}W^{s}(\varphi^{t}(x)) \ \ \text{and} \ \ W^{cu}(x)=\bigcup_{t\in\re}W^{u}(\varphi^{t}(x))$$
are invariant $C^{r}$-manifolds tangent to $E^{s}_{x}\oplus \phi(x)$ and $E^{u}_{x}\oplus \phi(x)$, respectively.\\

A special hyperbolic set, where fractal properties are well known, are the \textit {basic} sets, which means:
\begin{enumerate}
\item[(a)] the periodic orbit contained in $\Lambda$ are dense in $\Lambda$,
\item[(b)]$\phi^t|_{\Lambda}$ is transitive,
\item[(c)] There is an open set $U\supset \Lambda$ so that $\Lambda=\bigcap_{t\in\re}\phi^{t}(U)$.
\end{enumerate} 
The definition of hyperbolic sets (basic sets, Anosov) for diffeomorphisms is analog, as well as, the properties above for the stable and unstable manifolds are valid.\\
For diffeomorphisms on a surface, the basic sets also are called the \textit{horseshoe}.\\

\indent Some examples of Anosov flows are geodesic flows on unit tangent
bundles of compact Riemannian manifolds of negative curvature, and \textit{suspensions} of Anosov diffeomorphism (see Section \ref{main-results}).


\ \\
According to ergodic theory, there are  invariant probability measures for any  Anosov flows (diffeomorphisms). Thus, denote by $\mathcal{M}$ the set of that invariant probability measures. When an Anosov flow (diffeomorphisms) preserves a probability measure  $\tilde{m} \in\mathcal{M}$  absolutely continuous with respect to the normalized Lebesgue measure, we call it \emph{conservative}. In this work, we focus on three-dimensional Anosov flows and two-dimensional Anosov diffeomorphisms. 

\section{Separation Lemma and Hyperbolic set}\label{SL}

In this section, we will show that is possible to enclose any  hyperbolic  set of $\phi^t$ into a finite number of tubular neighborhood generated by GCS (Good Cross Sections) pairwise disjoint (see Definition \ref{D1GCS}). Using this GCS, for our basic set $\Lambda$, we can construct a basic set with Hausdorff dimension bigger than $1$, for the Poincar\'e map, restricted to the union of such GCS.  Also, we can  conclude that all hyperbolic sets of $\phi^t$ are one-dimensional. For this section, we can assume that $\Lambda$ is simply a hyperbolic set.
\subsection{Good cross-sections}\label{GCS}
The goal of this section is to present the Lemma \ref{L7} (Separation Lemma) which is a very important tool for the proof of theorem \ref{Theorem 1} and whose prove will be made in Appendix \ref{App - Separation}. \\
\indent Let us fix the following notation, we use $\mathcal{F}^{s}$ and $\mathcal{F}^{u}$ the strong stable and unstable  foliation,  \emph{i.e.},  $\mathcal{F}^{i}(x)=W^{i}(x)$ for $i=s, u$, which are continuous foliations of dimension one (not necessarily $C^1$-foliations). We also denote $\mathcal{F}^{s,u}_{\text{loc}}=W^{s,u}_{\text{loc}}$ the local stable (unstable) foliation.
\begin{Defi}\label{Def1} A $C^{0}$-surface $S$ is transverse to the flow $\phi^t$, if there are $\theta, r >0$ such that for every $z\in S$  the cone $C_z$ of angle $\theta$ centered in $\phi(z)$ with vertex at the point $z$ satisfies $C_z\cap B_{r}(z)\cap S=\{z\}$. 
\end{Defi}
\begin{Le}\label{Le5}
For each $x \in M$, let  $L$ be a $C^{1}$-embedded curve of dimension one, containing $x$ and $C^1$-transverse to the foliation $\mathcal{F}^{s}$, then the set
\begin{center}
$\ds S_{L}:=\bigcup_{z\in L}\mathcal{F}^{s}(z)$
\end{center}
contains a surface $S_{x}$, $C^{0}$-embedded, which contains $x$ in its interior. Moreover, if $L$ is $C^1$-transverse to the foliation $W^{cs}$ then, $S_{x}$ is $C^0$-transverse to the flow. 
\end{Le}
\begin{proof}
The first part of the theorem is by definition of $C^0$-foliation. For the second part, note that $L$ is $C^1$-transverse to $\mathcal{F}^{cs}$, then $L$ is $C^1$-transverse to $\phi^t$, moreover, the flow $\phi^{t}$ is $C^1$-transverse to $\mathcal{F}^{s}$, thus by the continuity of $\mathcal{F}^{s}_{\text{loc}}$ we can construct the surface $S_x$, $C^0$-transverse to $\phi^t$.
\end{proof}
In particular, taking $L=W^{u}_{\epsilon}(x)$ with $\epsilon$ given by the stable and unstable manifold theorem, we call $S_{L}:=S_{x}$. Note that an analogous lemma holds for the foliation $\mathcal{F}^{u}$.
\begin{R}\label{R1}
Note that the surface $S_x$ is a $C^0$-surface saturated by the foliation $\mathcal{F}^{s}$, therefore there is a homeomorphism $h\colon[0,1]\times[0,1]\to S_{x}$ such that the horizontal lines $[0,1]\times \eta$ are mapped to the stable sets $W^{s}(y,S_{x})=W^{s}(y)\cap S_{x}$. Therefore, we can define the stable-boundary, $\partial^{s}S_x$, of $S_x$, as being the image of $[0,1]\times\left\{0,1\right\}$ by the homeomorphism $h$ and the unstable-boundary, $\partial^{u}S_x$, of $S_x$ as being the image of $\left\{0,1\right\}\times [0,1]$ by the homeomorphism $h$. 
\end{R}
\indent From now on, unless otherwise stated, we consider  cross-section as the Lemma \ref{Le5}.

\begin{Defi}\label{D1GCS} Let $\Lambda$ be  a closed subset $M$.\\
We said that a compact cross-section $\Sigma$ is a \textbf{Good Cross-Section} (or simply GCS) for $\Lambda$ if 
\begin{center}
$d(\Lambda \cap \Sigma, \partial^{u}\Sigma)>0$
\ \ and \ \
$d(\Lambda \cap \Sigma, \partial^{s}\Sigma)>0$,
\end{center}
where $d$ is the intrinsic distance in $\Sigma$.
\end{Defi}
By compactness of $\Sigma$ and $\Lambda$, the above definition implies that there is $\delta>0$ such that  (cf. Figure \ref{fig:f7}).
\begin{equation}\label{delta-GCS}
d(\Lambda \cap \Sigma, \partial^{u}\Sigma)>\delta
\ \ \text{and} \ \
d(\Lambda \cap \Sigma, \partial^{s}\Sigma)>\delta.
\end{equation}
\begin{figure}[htbp]
	\centering
		\includegraphics[width=0.4\textwidth]{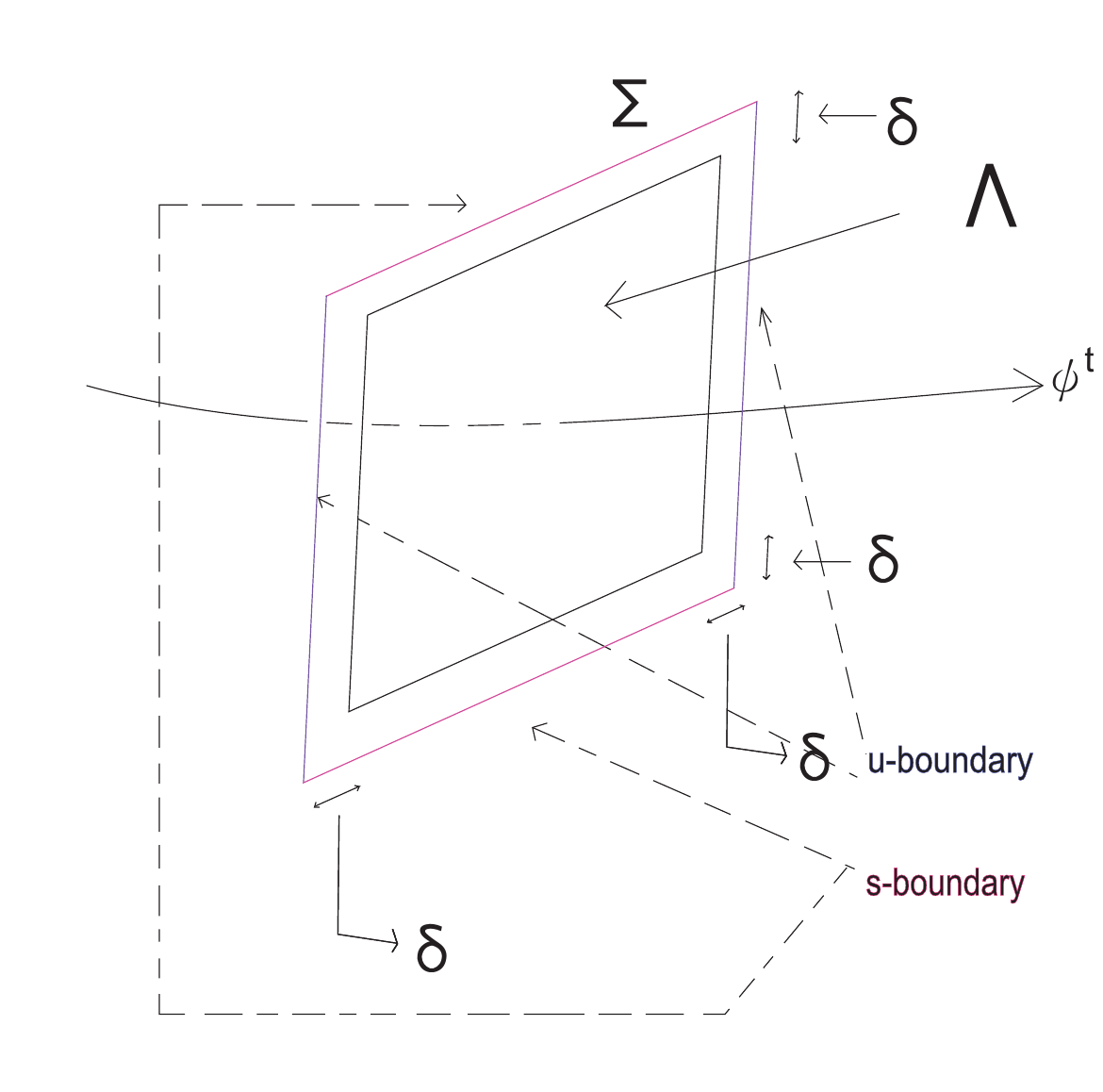}
	\caption{Good Cross-Section}
	\label{fig:f7}
\end{figure}
The Good Cross Sections  play an important role in reducing Theorem \ref{Theorem 1} to a two-dimensional problem since they allow to enclose the set $\Lambda$ in a tubular neighborhood  far from its boundary.  Moreover
\begin{R}\label{R2'}
Let $\Sigma$ be a GCS,  there are two GCS $\Sigma'$ and $\Sigma''$ such that  $$\Sigma'\subset \emph{int}(\Sigma), \, \, \,  \partial \Sigma'\cap \partial \Sigma=\emptyset \ \ \text{and} \ \ \Sigma\subset \emph{int}(\Sigma''), \,\, \,  \partial \Sigma\cap \partial \Sigma''=\emptyset.$$ Therefore, from now on, we can assume that if two GCS has nonempty intersection, then their interiors have nonempty intersection. 
\end{R}
The Good Cross-Sections may not exist in general, but if  $\Lambda\subsetneqq M$ is a hyperbolic set of a transitive three-dimensional Anosov flow, then  we will prove  that Good Cross Sections always exist (cf. Lemma \ref{Le6}). For this sake, we prove the following lemmas:

\begin{Le}\label{L3GCS} Let $\Lambda\subsetneqq M$ a hyperbolic set for a transitive three-dimensional Anosov flow on $M$. Then, for any $x\in \Lambda$ there exist points $x^{+}\notin \Lambda$ and $x^{-} \notin \Lambda$ in distinct connected components of $W^{s}(x)\setminus\left\{x\right\}$.
\end{Le}
\begin{proof}
Note that, for three-dimensional Anosov flow, the stable manifold is one-dimensional. 
Let $x\in \Lambda$, by contradiction, assume that there is a segment of the strong stable manifold entirely contained in $\Lambda$ and containing $x$ in the interior, we called by $\zeta$ this segment. Without loss of generality, we can assume that $W^{s}_{loc}(x)\subset \zeta$. Now take $t_{k}$ a sequence such that $t_{k}\to \infty$ as $k\to \infty$. Then, as $\Lambda$ is a compact invariant set, we can assume that $\phi^{-t_{k}}(x)\to y\in \Lambda$ as $k\to \infty$. The point $y$ satisfies:\\
\ \\
{Claim:} $W^{s}(y)\subset \Lambda$. 
\begin{proof}[{\emph{Proof of Claim}}]
Let $z\in W^{s}(y)$, as $W^{s}(y) =\bigcup_{t\geq 0}\phi^{-t}\left(W^{s}_{loc}(\phi^{t}(y)\right)$, then there is $T\geq 0$, such that $\phi^{T}(z)\in W^{s}_{loc}(\phi^{T}(y))$. Then by Stable Manifold Theorem $W^{s}_{loc}(\phi^{T}(y))$ is accumulated by points of $W^{s}_{loc}(\phi^{(-t_{k}+T)}(x))$, for large enough $k$. Let $k$ be sufficiently large such that  $(-t_{k}+T)<0$ and $W^{s}_{loc}(\phi^{(-t_{k}+T)}(x))\subset \phi^{(-t_{k}+T)}(\zeta)\subset \Lambda$, since $\Lambda$ is an invariant set and $\zeta\subset \Lambda$. Hence as $\Lambda$ is closed, we have that $W^{s}_{loc}(\phi^{T}(y)) \subset\Lambda$ which  implies that $z\in \Lambda$, and therefore completes proof of claim.
\end{proof}

\noindent The above claim implies that $\Lambda \supset W^{cs}(y)=\bigcup_{t\in \mathbb{R}}W^{s}(\phi^{t}(y))$. In fact: \\
Let $w\in W^{cs}(y)$, then there is $t_{0}\in \mathbb{R}$ such that $w\in W^{s}(\phi^{t_{0}}(y))$. Hence, there is $T\geq 0$ such that $\phi^{T}(w) \in W^{s}_{\epsilon}(\phi^{T+t_{0}}(y))$. Since  $\phi^{T+r}(w) \in W^{s}_{K\epsilon e^{-\lambda r}}(\phi^{T+r+t_{0}}(y))$ for $r\geq 0$, then we can assume that $T+t_{0}>0$. Thus, $$\phi^{-t_{0}}(w)=\phi^{-(T+t_{0})}(\phi^{T}(w))\in \phi^{-(T+t_{0})}\left(W^{s}_{\epsilon}(\phi^{T+t_{0}}(y))\right)\subset W^{s}(y) \subset \Lambda.$$ Since $\Lambda$ is invariant, then we have that   $W^{cs}(y)\subset \Lambda$.\\ So, the $\phi^t$ is transitive, then $M=\overline{W^{cs}(y)}\subset \Lambda$ (cf. \cite{K}),  which  provides a contradiction. Thus, we concluded the proof of the lemma.
\end{proof}
\noindent Analogously we have,
\begin{Le}\label{L4GCS}Let $\Lambda\subsetneqq M$ a hyperbolic set for a transitive three-dimensional Anosov flow on $M$. Then,
for any $y\in \Lambda$ there are points $y^{+}\notin \Lambda$ and $y^{-}\notin \Lambda$ in distinct connected components of $W^{u}(x)\setminus \left\{x\right\}$.
\end{Le}
\begin{proof}
Similar to the proof of  Lemma \ref{L3GCS}. 
\end{proof}
\begin{Le}\label{Le6}Let $\Lambda\subsetneqq M$ a hyperbolic set for a transitive three-dimensional Anosov flow on $M$. Then, for every $x\in \Lambda$ there is a Good Cross-Section $\Sigma_x$ at $x$ with $\Sigma_x\subset S_x$.
\end{Le}
\begin{proof}
Fix $\epsilon>0$ as in the stable and unstable  manifold theorem, and consider the cross-section $\Sigma_{x}$ given by the Lemma \ref{Le5} containing the segments of $W^{s}_{\epsilon}(x)$ and $W^{u}_{\epsilon}(x)$ and the point $x$ in its interior. By Lemma \ref{L3GCS} and Lemma \ref{L4GCS}, we may find points $x^{\pm}\notin \Lambda$ in each of the connected components of $(W^{s}_{\epsilon}(x)\cap\Sigma_{x})\setminus \{x\}$ and points $z^{\pm}\notin \Lambda$ in each of the connected components of  $(W^{u}_{\epsilon}(x)\cap\Sigma_{x})\setminus\{x\}$. Since $\Lambda$ is closed, there are neighborhoods $V^{\pm}$ of $x^{\pm}$ and $V_{1}^{\pm}$ of $z^{\pm}$ respectively disjoint from $\Lambda$, (cf. Figure \ref{fig:figura2novo}).

\begin{figure}[htbp]
	\centering
		\includegraphics[width=0.3\textwidth]{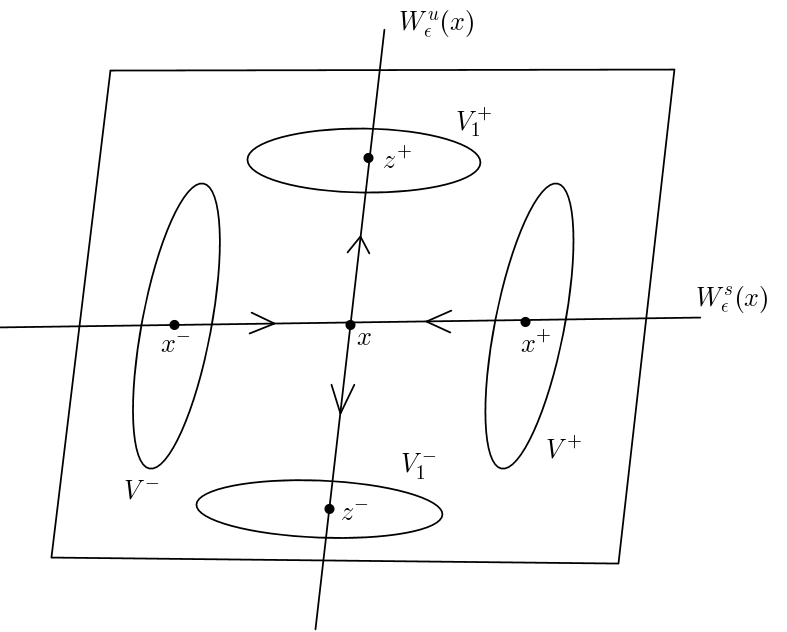}
		\caption{The first step to construct GCS for $x\in \Lambda$}
	\label{fig:figura2novo}
\end{figure}
\ \\
In Figure \ref{fig:figura2novo},  $V^{\pm}$, $V_{1}^{\pm}$ may  enclose a region homeomorphic to a square, in this case, there is nothing to be done. Otherwise, we prove that we can obtain open sets in the cross-section which does not intersect $\Lambda$ and enclose a region homeomorphic to a square. Indeed, 
let $t_{k}$ be a sequence such that $t_{k} \rightarrow +\infty$ as $k\rightarrow +\infty$ and $\phi^{t_k}(x)\rightarrow y \in \Lambda$ as $k\rightarrow +\infty$, then by Lemma \ref{L3GCS}, there are $y^{\pm}$ in each of the connected components of $ W^{u}_{\epsilon}(y)$ and $y^{\pm}\notin \Lambda$, so  there are neighborhoods $J^{\pm}$ of $y^{\pm}$, respectively, with $J^{\pm}\cap \Lambda=\emptyset$.\\
The stable and unstable manifold theorem provides the following properties for $y$. 
\begin{itemize}
\item[(i)]There is a  neighborhood  $U_y$ of $y$ such that 
$$W_{\epsilon}^{u}(z)\cap J^{\pm}\neq \emptyset, \, \, \text{for all} \, \, z\in U_y.$$
\item[(ii)] There is $\delta>0$ such that $W^{s}_{\delta}(y^{\pm})\subset J^{\pm}$, respectively. 
\end{itemize}
Note that, there is  $k_0$ such that $\phi^{t_k}(x)\in U_y$, for all $k\geq k_0$. Then by the property (i) and the stable and unstable  manifold theorem (increasing $k_0$, if necessary), we have that there is $w_k^{\pm}\in W_{\epsilon}^{u}(\phi^{t_k}(x))\cap J^{\pm}$ which converge to  $y^{\pm}$, respectively. Then, again by the  stable and unstable manifold theorem and property (ii), there are $\delta_0\leq \dfrac{\delta}{2}$ and $k_{1}\geq k_0$ such that 
\begin{equation}\label{NE2}
W^{s}_{\delta_0}(w_{k}^{\pm})\subset J^{\pm}\, \, \,  \text{and} \, \, \, W^{s}_{\epsilon}(\phi^{-t_k}(w_k^{\pm}))\subset \phi^{-t_k}(W^{s}_{\delta_0}(w_k^{\pm})),
\end{equation}
whenever $k\geq k_1$.
Observe that $$d(x, \phi^{-t_k}(w_k^{\pm}))=d(\phi^{-t_k}(\phi^{+t_k}(x)), \phi^{-t_k}(w_k^{\pm}))\leq C\lambda^{t_k}d(\phi^{t_k}(x), w_k^{\pm})\leq C\lambda^{t_k}\epsilon.$$
Thus, for $k\geq k_1$ large enough, the last inequality and the stable and unstable theorem implies that $W^{s}_{\epsilon}(\phi^{-t_k}(w_k^{\pm}))$ is $C^0$-close to $W^{s}_{\epsilon}(x)$ and therefore satisfies  $W^{s}_{\epsilon}(\phi^{-t_k}(w_k^{+}))\cap V^{\pm}\neq \emptyset$ and $W^{s}_{\epsilon}(\phi^{-t_k}(w_{k}^{-}))\cap V^{\pm}\neq \emptyset$.

 Moreover, the relation (\ref{NE2}) implies that $W^{s}_{\epsilon}(\phi^{-t_k}(w_k^{\pm})) \subset \phi^{-t_k}(J^{\pm})$ and therefore, since $J^{\pm}\cap \Lambda=\emptyset$, then $W^{s}_{\epsilon}(\phi^{-t_k}(w_k^{\pm}))\cap \Lambda = \emptyset$, for $k\geq k_1$.\\
 Consider  $V_{k}^{\pm}$  a neighborhood of $W^{s}_{\epsilon}(\phi^{-t_k}(w_k^{\pm}))$, respectively, such that  $V_{k}^{\pm}\cap \Lambda=\emptyset$.  As we know that $\phi^{-t_k}(w_k^{\pm})\in W^{u}_{\epsilon}(x)$, then  $V^{\pm}\cap \Sigma_x$ and $V_{k}^{\pm}\cap \Sigma_x$ enclose a region homeomorphic to a square and such region contains a Good Cross-Section $\Sigma_k:=\Sigma_x$, as we wish. 

\end{proof}

\begin{R}\label{R3}
It is worth note that, the stable boundary of \, $\Sigma_{k}$, $\partial^{s}\Sigma_k$, is equal to  $W^{s}_{\epsilon}(\phi^{-t_k}(w_k^{\pm}))$, which  converge $($in the $C^0$-topology$)$ to $W^{s}_{\epsilon}(x)$. Thus, we can state that the cross- section of the above lemma can be taken such that the stable boundary  as close as you want to   $W^{s}_{\epsilon}(x)$. Similarly, the Good Cross-Section of the above lemma can be constructed using unstable saturation. 
\end{R}

\noindent This kind of cross-section has good properties. Before showing  the properties of such sections, remember that they  are $C^{0}$-sections, so we will need some definitions.
\begin{Defi}\label{D2}
We say that a continuous curve $\xi\subset \re^2$ is $\theta$-transverse in neighborhood of radius $r$ to one-dimensional foliation $\mathcal{F}$ (with $C^{1}$-leaves) in $\re^2$, if for any $z\in \xi\cap \mathcal{F}_{z}$ $($here $\mathcal{F}_{z}$ is the leaf containing $z$\,$)$ there is a cone $C$ with vertex at the point $z$ such that $\xi\cap B(z,r)\subset C$ and the angle $\angle(v, T_{z}\mathcal{F}_{z})\geq \theta$ for every tangent vector $v$ at the point $z$ contained in the cone $C$. 
\end{Defi}

\noindent As the section $\Sigma_x$ is saturated by the foliation $\mathcal{F}^{s}$, then we call $h_{x}$ the homeomorphism given in Remark \ref{R1}, then we have the following definition: 
\begin{Defi}\label{D3}
We say that a continuous curve $\zeta\subset \Sigma_x$ is transverse to foliation $\mathcal{F}^{s}$, if there are $\theta$ and $r$ such that \, $h_{x}^{-1}(\zeta)$ is $\theta$-transverse in a neighborhood of radius $r$ to the foliation $\{h_x^{-1}(\mathcal{F}^{s}(z)\cap \Sigma_x): z\in \mathcal{F}_{loc}^{u}(x)\}$. 
\end{Defi}

\begin{Pro}\label{P3}
Given $x,y \in \Lambda$, such that there is a $C^{0}$-curve $\zeta\subset \emph{int}\,\Sigma_{x}\cap \emph{int}\,\Sigma_{y}$. If $\zeta$  intersects transversely to foliation $\mathcal{F}^{s}$, then $\emph{int}\Sigma_{x}\, \cap \,\emph{int}\Sigma_{y}$ is an open set of $\Sigma_{x}$ and $\Sigma_{y}$.
\end{Pro}
\begin{proof}
Since $\zeta\subset \text{int}\,\Sigma_{x}\cap \text{int}\,\Sigma_{y}$ a $C^{0}$-curve transverse to $\mathcal{F}^{s}$. Then for all $z\in \zeta$, there are $x^{\prime} \in W^{u}_{\epsilon}(x)$ and $y^{\prime} \in W^{u}_{\epsilon}(y)$ such that $z\in W^{s}(x^{\prime})\cap \Sigma_{x}$ and $z\in W^{s}(y^{\prime})\cap \Sigma_{y}$, thus  $W^{s}(x^{\prime})=W^{s}(y^{\prime})$. Therefore, there is $\delta>0$ such that the set 
\begin{center}
$\ds B=\bigcup_{z\in \zeta}W^{s}_{\delta}(z)\subset \text{int}\,\Sigma_{x} \cap \text{int}\,\Sigma_{y}$.
\end{center}
Thus, we concluded the proof of proposition.
\end{proof}
\subsubsection{Separation of GCS}\label{Separation of GCS}

\noindent By Lemma \ref{Le6}, at each point of $x\in \Lambda$, we can find a Good Cross-Section $\Sigma_{x}$. Since $\Lambda$ is a compact set, then for $\gamma>0$, there are a finite number of points $x_{i}\in \Lambda$, $i=1,\dots,l$ such that 
\begin{equation}\label{E6GCS}
\ds \Lambda\subset  \bigcup^{l}_{i=1}\phi^{(-\gamma,\gamma)}(\text{int}\,{\Sigma_{i}}):=\bigcup^{l}_{i=1}U_{\Sigma_{i}},
\end{equation}
where $\Sigma_{i}:=\Sigma_{x_i}$.\\

The main goal of this section is the following lemma, which has  very technical proof that will be presented in the Appendix \ref{App - Separation}.

\begin{Le}\label{L7}
There is $m\in \mathbb{N}$ and GCS $\widetilde{\Sigma}_{i}$, $i=1,\dots, m$ such that 

\begin{equation}\label{eq6}
\ds\Lambda\subset \bigcup^{m}_{i=1}\phi^{(-2\gamma,2\gamma)}(\emph{int}\,\widetilde{\Sigma}_i)
\end{equation}
with $\widetilde{\Sigma}_i\cap \widetilde{\Sigma}_j=\emptyset$.
\end{Le}

The above lemma state that the Good Cross-Sections in (\ref{E6GCS}) can be taken pairwise disjoint and therefore  reduce our problem to the study the  Lagrange and Markov spectra of the  map of first return (Poincar\'e map) on the union of these sections (see Section \ref{Hyp of PM} and Section \ref{RS}).

\begin{R}\label{R11}
Since $C^{\infty}$-topology is dense in $C^{0}$-topology, 
from now on, we can assume without loss of generality,  that there are $C^{\infty}$-GCS, $\Sigma_{i}$, pairwise disjoint which satisfies the  \emph{Lemma \ref{L7}}. 

\end{R}

We ended this section by announcing an immediate consequence of  Lemma \ref{L7} and the definition of GCS, which will be used in Section \ref{Sec 4.3} to construct a basic set for  geodesic flows in pinched negative curvature. 
\begin{C}\label{C3}
Any hyperbolic set of a three-dimensional transitive Anosov flow  has topological dimension $1$.
\end{C}
\subsection{Poincar\'e Map}\label{PM}

 Let $\Xi=\bigcup_{i=1}^{m} \Sigma_{i}$ be a finite union of cross-sections to the flow $\phi^{t}$ given by  Remark \ref{R11}, which are pairwise disjoint. Sometimes, abusing of notation,  we consider $\Xi=\{\Sigma_1,\cdots, \Sigma_l\}$. Let ${\cR}\colon \Xi \to \Xi$ be a Poincar\'e map, that is, the map of first return to $\Xi$, ${\cR}(y)=\phi^{t_{+}(y)}(y)$, where $t_{+}(y)$ corresponds to the first time that the positive orbits of $y\in \Xi$ encounter $\Xi$. 

\subsubsection{Hyperbolicity of Poincar\'e Map}\label{Hyp of PM}
The hyperbolicity on $\Lambda$ induces hyperbolicity of $\Lambda\cap \Xi$. More precisely, if we denote by $\Delta:=\bigcap_{n\in \mathbb{Z}}{\cR}^{-n}(\Xi))$, then 
\begin{Le}\label{LHPM}
The set $\Lambda\cap \Xi$ is hyperbolic for ${\cR}$ and satisfies 
$$\Lambda\cap \Xi\subset  \Delta.$$
\end{Le}

\noindent We used some arguments find at  \cite[ch. 6]{VP} to do the proof of Lemma \ref{LHPM}, which will be presented in  Appendix \ref{PHPM}. 

\subsubsection{Hausdorff Dimension of Hyperbolic set of $\mathcal{R}$}\label{SEC1.5GCS}

In this section, we estimate the Hausdorff dimension of $\Lambda$ using the Hausdorff dimension of $\Delta$.
\begin{Le}\label{L13GCS}
The set $\Lambda$ satisfies $$\Lambda\subset\bigcup_{t\in\mathbb{R}}\phi^{t}\Big( \bigcap_{n\in \mathbb{Z}}{\cR}^{-n}(\Lambda\cap \Xi ) \Big )=\bigcup_{t\in\mathbb{R}}\phi^{t}(\Lambda\cap \Xi)
\subset\bigcup_{t\in\mathbb{R}}\phi^{t}(\Delta).$$
\end{Le}
\begin{proof}
Remember that $\ds\Lambda\subset \bigcup^{l}_{i=1}U_{\Sigma_{i}}$, where $U_{i}=\phi^{(-2\gamma,2\gamma)}(\text{int}\, \Sigma_{i})$. Let $z\in \Lambda$, then there is $t_{z}$ such that $z=\phi^{t_{z}}(x)$ with $x\in \text{int}\,\Sigma_{i}$ for some $i$. This implies that $x\in \Lambda\cap \Xi$ and therefore, ${\cR}(x)\in \text{int} (\Sigma_{j})$ for some $j$, so  ${\cR}(x)\in \text{int}\, (\Xi)$. Analogously, ${\cR}^{n}(x) \in \text{int}\,(\Xi)$, \emph{i.e.}, $\mathcal{R}^{n}(x)\in \Lambda\cap \Xi$ for all $n\in\mathbb{Z}$. Hence, $x\in \bigcap_{n\in \mathbb{Z}}{\cR}^{-n}(\Lambda \cap \Xi)$, therefore $z\in \phi^{t_z}\left(\bigcap_{n\in \mathbb{Z}}{\cR}^{-n}(\Lambda \cap \Xi)\right)$.
\end{proof} 
\begin{Le}\label{L14GCS}
The Hausdorff dimension of $\Lambda\cap \Xi$ and  $\Delta$ satisfies, $$HD(\Delta)\leq  HD\left(\Lambda \cap \Xi)\right)+1.$$
\end{Le}
\begin{proof}
Take a \emph{bi}-infinite sequence $$\cdots<t_{-k}<t_{-k+1}<\cdots<t_{0}<t_{1}<\cdots t_{k}<\cdots$$ such that $\left|t_{k}-t_{k+1}\right|<\alpha$ with $\alpha$ sufficiently small, then
\begin{center}
$\Lambda \subset \bigcup^{+\infty}_{k=-\infty}\phi^{\left[t_{k},t_{k+1}\right]}(\Lambda\cap \Xi):=\bigcup^{+\infty}_{k=-\infty}A_{k}.$
\end{center}
Then, $HD(\Lambda)\leq \sup_{k}HD(A_{k})$. Moreover, if  $\alpha$ is small enough, the map 
\begin{eqnarray*}
\psi_{k}:&\left(\Lambda\cap \Xi \right)\times\left[t_{k},t_{k+1}\right]&\longrightarrow A_{k} \ \ \text{defined  by} \\
&(x,t)&\longmapsto\phi^{t}(x)
\end{eqnarray*}
is Lipschitz. Therefore, if we call $I_{k}=\left[t_{k},t_{k+1}\right]$, it is  easy to see that 
$$HD(A_{k})\leq HD\left((\Lambda\cap\Xi)\times I_{k}\right)\leq HD\left(\Lambda\cap\Xi\right)+D(I_{k}),$$ 
where $D$ is an upper box-counting dimension of $I_{k}$. It is easy to see that  $D(I_{k})=1$ (cf. \cite{Falconer}).
Thus,
$$HD(\Lambda)\leq \sup_{k}\ HD(A_{k})\leq HD\left( \Lambda\cap\Xi\right) + 1.$$
\end{proof}
\begin{C}\label{C-L13GCS}
If $HD(\Lambda)>2$, then  $HD(\Lambda\cap\Xi)>1$.
\end{C}

\section{Lagrange and Markov Spectrum}\label{LMS}
In this section, we prove the Theorem \ref{Theorem 1}. In this direction, we will prove an equivalent version (Theorem \ref{Theorem 2}), which reduces the problem to find non-empty interior for the Lagrange and Markov spectrum for discrete dynamical systems in dimension two. 
\subsection{Regaining the Spectrum}\label{RS}

\indent The dynamical Lagrange and Markov spectra of $\Lambda$ and $\Delta$ are related in the following way.
Given a function $F\in C^{s}(M,\mathbb{R})$, $s\geq 1$, let us denote by $f =\textrm{max}F_{\phi}\colon D_{\cR}\to\re$ the function
$$\textrm{max} F_{\phi}(x):=\max_{0\leq t \leq t_{+}(x)}F(\phi^{t}(x)),$$
where $D_{\cR}$ is the domain of $\cR$ and $t_{+}(x)$ is such that $\cR(x)=\phi^{t_{+}(x)}(x)$.\\
It is not difficult to show that 
$$\limsup_{n\to +\infty}f(\cR^n(x))=\limsup_{t\to + \infty}F(\phi^t(x))$$
and 
$$\sup_{n\in \mathbb{Z}}f(\cR^n(x))=\sup_{t\in \re}F(\phi^t(x))$$
for all $x\in \Delta$. In particular,

\begin{equation}\label{Continuos and Discrete}
L(\phi, \Lambda, F)=L(\cR, \Delta, f) \ \ \text{and} \ \ M(\phi, \Lambda, F)=M(\cR, \Delta, f).
\end{equation}

\begin{R}\label{RND}
$f =\textrm{max}F_{\phi}$ might not be $C^1$ in general.
\end{R}
\begin{R}
It is worth noting that given a vector field $X$ close to $\phi$, the  Poincar\'e map $\cR_{X}$ of the flow of $X$  is 
 defined in the same cross-sections where $\cR$ is defined. 
\end{R}
Thus, the relations (\ref{Continuos and Discrete}) reduces Theorem \ref{Theorem 1}  to the following theorem:

\begin{T}\label{Theorem 2}Let $\phi$ be a vector field, such that ${\phi}^{t}$ is a conservative Anosov flow, which has a basic set $\Lambda$ with Hausdorff dimension bigger than $2$, then $C^2$-arbitrarily close to $\phi$ there is an open set $\mathcal{W}\subset \mathfrak{X}_{w}^{2}(M)$, such that for any $X\in \mathcal{W}$, if $\Delta_{X}$ is the  hyperbolic continuation of  $\Delta$ by the Poincar\'e map $\cR_X$, one can find a dense and  $C^{2}$-open subset  $\mathcal{U}_{X,\Lambda}\subset C^{2}(M,\re)$, so that 
$$\emph{int}\, M(\cR_{X},\Delta_{X}, \textrm{max} F_{X})\neq \emptyset \, \, \text{and}\, \,  \, \emph{int}\, L(\cR_{X},\Delta_{X}, \textrm{max} F_{\phi})\neq \emptyset,$$
whenever $F\in \mathcal{U}_{X,\Lambda}$. Moreover, the above statement holds persistently, \emph{i.e.}, for any $Y\in \mathcal{W}$, it holds for any $(F,X)$ in a suitable neighborhood of $\mathcal{U}_{Y,\Lambda}\times \{Y\}$ in $C^{2}({M,\re})\times \mathfrak{X}_{w}^{2}(M)$. 
\end{T}

\subsection{Family of Perturbations}\label{Family of Pert}
\noindent In Section \ref{SL} has been proven that there is a finite number of $C^{\infty}$-GCS, $\Sigma_{i}$ pairwise disjoint and  such that
the Poincar\'e map $\cR \colon \Xi\to \Xi,$ (first return to $\ds \Xi$)
where $\ds \Xi:=\cup^{l}_{i=1}\Sigma_{i}$  satisfies:
\begin{itemize}
 \item $\bigcap_{n\in \mathbb{Z}}\mathcal{R}^{-n}(\Xi):=\Delta$ is a basic set for $\cR$, since $\Lambda$ is a basic set for $\phi^t$.
\item If $HD(\Lambda)>2$, then $\ds HD\left(\Delta\right)> 1$,
\end{itemize}

The main goal of this section is to construct a family of perturbations of $\phi$, which produces  perturbations on $\cR$ so that we can  apply the techniques of \cite{RM} (cf. Appendix \ref{SIRCS} and \cite{MY}).
\begin{R}\label{R_pert}
From now on, we will consider vector fields $X\in \mathfrak{X}^{2}_{\omega}(M)$, $C^2$-sufficiently close to $\phi$ such that: If we denote  by $\cR_{X} \colon \Xi\to \Xi$ the Poincar\'e map associated to $X$, then 
\begin{enumerate}
\item[$1.$] There exists the hyperbolic continuation $\Delta_{X}$  of $\Delta$ by the map  $\cR_{X}$.
\item[$2.$] $HD(\Delta_{X}) > 1$, since  the Hausdorff dimension of the basic sets, is continuous for  $C^2$-diffeomorphisms on $\Xi$ \emph{(cf. \cite{PT})}.
\end{enumerate} 
\end{R}
\subsubsection{First Perturbation for The Birkhoff Invariant}\label{SBI}
\indent Since the flow is conservative, then the Poncar\'e map $\mathcal{R}$  is a conservative diffeomorphism. Thus, to describe the family of perturbations of $\cR$ given in \cite{MY} to apply the techniques of \cite{RM},  we need that the Birkhoff invariant will be non-zero in a periodic point of $\cR$ (cf. Appendix \ref{BI}  and \cite[Section 4.3]{MY1}).
As our perturbations are in the conservative world and we are free to perturb the vector field $\phi$ in such a way that the Birkhoff invariant be non-zero for some periodic orbit of the new Poincar\'e map. In other words:

\begin{R}\label{RBI}
We can assume, from now on, that the Poincar\'e map $\cR$ associated to the flow $\phi$ has the property that the Birkhoff invariant is non zero for some periodic orbit \emph{(see, {Appendix \ref{BI}})}.
\end{R} 
\subsubsection{Family of Perturbations {with the Property V}}\label{Realiz of the Pert}
The central goal of this section is to do  small conservative  perturbations of $\phi$, in order to produce a family of perturbations of $\cR$ with good properties, which allow the use of the techniques of \cite{RM}, more specifically, the \emph{property $V$} which will be explained in Appendix  \ref{SIRCS}.
Therefore, the following three lemmas focus on this goal. 
\begin{Le}\label{Le 6.2}
Given $\mathcal{V}$ a $C^{r}$-neighborhood of $\phi$ and $p\in \Delta \cap\Sigma$ with $\Sigma\in \Xi$. Let $U$ be a neighborhood of $\phi^{\frac{t_{+}(p)}{2}}(p)$, then  there exists  a  conservative vector field $X\in \mathcal{V}$ such that:
\begin{enumerate}
\item[$(1)$]$X\equiv\phi$ outside of $U$,
\item[$(2)$] There is $\tau>0$, such that $X\equiv\phi$ outside of a subset of $U$ of the form 
$X^{[0,\tau]}(\Sigma_0)=\{X^{t}(x):x\in \Sigma_0, \, 0<t<\tau\}$, where $\Sigma_0$ is a  neighborhood of $\phi^{\frac{t_{+}(p)}{2}}(p)$ in $\phi^{\frac{t_{+}(p)}{2}}(\Sigma)$,
\item[$(3)$] The map  $\mathcal{R}_{X}$  satisfies  $\mathcal{R}_{X}(p)\neq\mathcal{R}(p)$.
\end{enumerate}
\end{Le}

\noindent The proof of this lemma is an immediate consequence of the two lemmas below. The first is about conservative trivialization and the second is about local conservative perturbations. 
\begin{Le}{\emph{\cite[Lemma\, 3.4 \, (\text{Conservative flow box theorem})]{Bessa1}}}\label{Le 6.3}\ \,\\
Let $X\in \mathfrak{X}^{r}_{\omega}(M)$, $p$ be a regular point of the vector field and $\Sigma$ a cross-section of $X$ which  contains $p$, then there exists a $C^{\infty}$-coordinate system $\alpha\colon U\subset M \to \re^3$ with $\alpha(p)=0$ and such that 
\begin{enumerate}
\item[$(a)$] $\alpha_{*}X=(0,0,1)$,
\item[$(b)$]$\alpha_{*}\omega=dx\wedge dy \wedge dz$,
\item[$(c)$]$\alpha(U\cap \Sigma)\subset\{z=0\}.$
\end{enumerate}
\end{Le}

To next lemma, let $B_{\delta}(x,y)\subset \re^2$ be the open ball of center $(x,y)$ and radius $\delta$. Similarly, $\overline{B}_{\delta}(x,y)$ denotes the closed ball. If $C$ is the cylinder $\partial B_{\delta}(x,y)\times [0,h]\subset \re^3$ and $0<\beta<\delta$, we define of neighborhood of $C$ as
$$A_{\beta}(C)=\left( B_{\delta+\beta}(x,y)\setminus \overline{B}_{\delta-\beta}(x,y)\right) \times [0,h]\subset \re^3.$$ 
and call it \textit{cylinder ring} with center at $C$ and radius $\beta$.
\begin{Le}{\emph{\cite[Lemma 3.2]{CandO}}}\label{Le 6.4}\\
Let $X\colon \re^3\to \re^3$ be the constant vector field defined by $X(x,y,z)=(0,0,1)$. Consider the cylinder $C=\partial B_{\delta}(0,0)\times [0,h]\subset \re^3$, $\delta>0$, $h>0$, and points $p\in \partial B_{\delta}(0,0)\times \{0\}$ and $q\in \partial B_{\delta}(0,0)\times \{h\}$. Let $\theta$ be the angle between the vector $p-(0,0,0)$ and $q-(0,0,h)$. Given $0<\beta<\delta$ there exists a $C^{\infty}$- vector field $Z$ on $\re^3$ with the following properties: 
\begin{enumerate}
\item[$(a)$] $Z$ preserves the canonical volume form $dx\wedge dy \wedge dz$,
\item[$(b)$] $Z\equiv X$ outside the cylinder ring $A_{\beta}(C)$,
\item[$(c)$]The positive orbit of $p$, with respect to $Z$, contains $q$,
\item[$(d)$]  Given $r\in \mathbb{N}$ and $\epsilon>0$, if $|\theta|$ is small enough, then ${\Vert Z-X\Vert}_{r}< \epsilon$, where $\Vert\cdot \Vert$ denotes the $C^r$ norm on the set of $C^r$ vector fields.
\begin{figure}[hbtp]
\centering
\includegraphics[scale=1.2]{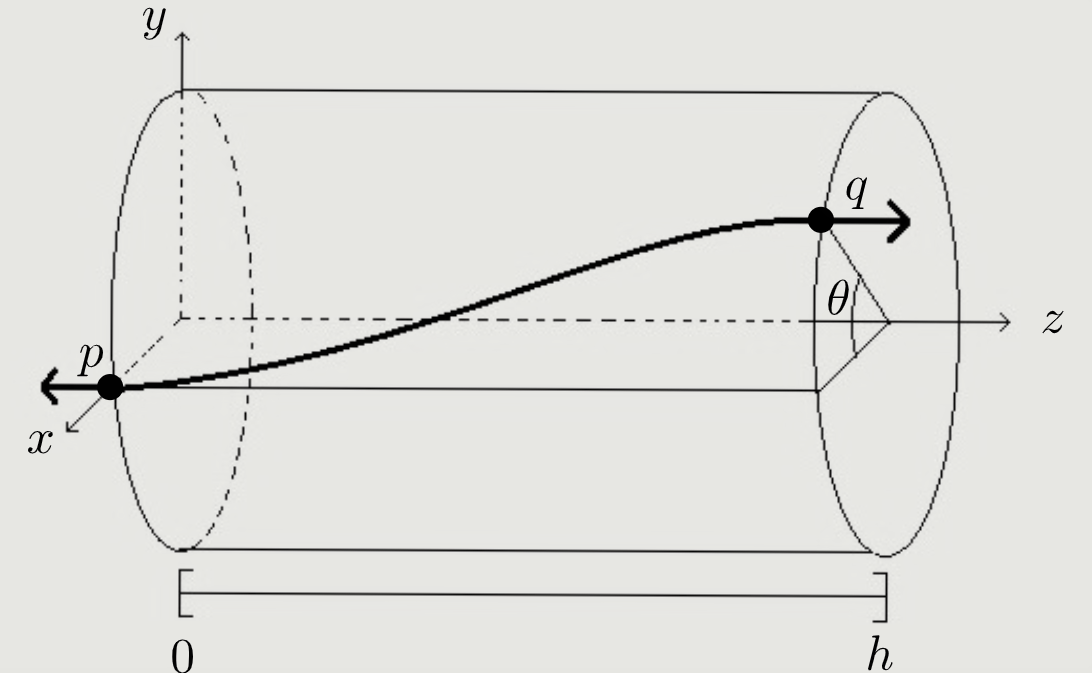}
\caption{The motion of the orbit }
\end{figure}

\end{enumerate}
\end{Le}
\begin{proof}[\emph{\textbf{Proof of Lemma \ref{Le 6.2}}}]
By Lemma \ref{Le 6.3}, we can consider a coordinates systems $\alpha\colon U'\to V$ in a neighborhood  $U'\subset U$ of $\phi^{\frac{t_{+}(p)}{2}}(p)$, with $\alpha(p)=(0,0,0)$, $\alpha_{*}\phi=(0,0,1)$, $\alpha_{*}\omega=dx\wedge dy\wedge dz$ and $\alpha(\phi^{\frac{t_{+}(p)}{2}}(\Sigma)\cap U')\subset \{z=0\}$. Let $\beta , \delta >0$, $0<h<1$ and $q\in\{z=0\}$ such that 
the solid cylinder $\overline{B}_{\beta+\delta}(q)\times [0,h]\subset V$ and $(0,0)\in \partial B_{\delta}(q)$.\\
Consider now the cylinder ring $A_{\beta}(C)$ defined by $\beta , \delta , h$ and $q$. Let $\theta$ be a small angle, and let $q'\in \partial B_{\delta}(q)\times \{h\}$, such that the angle between $(0,0,0)-q$ and $q'-(q,h)$ is equal to $\theta$. Now we may apply the perturbation Lemma \ref{Le 6.4} at the
cylinder $C=\partial B_{\delta}(q)\times[0, h]$ to join $0$ to $q'$ and obtain a vector field $Z$ on $V$ such that:
\begin{enumerate}
\item[$(a)$] $Z$ preserves the canonical volume form,
\item[$(b)$] $Z\equiv (0,0,1)$ in $V\setminus A_{\beta}(C)$,
\item[$(c)$] The positive orbit of $0$  respect to $Z$, contains $q'$.
\end{enumerate}
\indent \, \, \, Let us define the vector field $X$ in $M$ in the following way: $X\equiv \phi$ outside of $U'$ and 
$X=\alpha_{*}(Z)$ in $U'$. Note that $X$ is $C^{r}$ satisfies (1) and taking $\theta$ sufficiently small, we may assume $X\in \mathcal{V}$.\\
In order to prove (2), we consider $\Pi_0\subset V\cap\{z=0\}$ a compact neighborhood of
the origin contained in $\alpha(\phi^{\frac{t_{+}(p)}{2}}(\Sigma)\cap U'))$. Then just take $\Sigma_0=\alpha^{-1}(\Pi_{0})$ and $\tau=\sup \{t>0: \alpha(X^{t}(x))\in \Pi_{0}\times [0,h], x\in \Sigma_0\}$. The item (3) is an immediate consequence of the properties  (b) and (c) of the filed $Z$.
\end{proof}
\indent The proof of Lemma \ref{Le 6.2}, implies that for $p\in \Delta$ and  every small $\theta$ there is $X_{\theta}^{p}\in \mathfrak{X}^{r}_{\omega}(M)$ such that the Poincar\'e map ${\mathcal{R}}_{X_{\theta}^{p}}$ associated to $X_{\theta}^{p}$ satisfies $\mathcal{R}_{X_{\theta}^{p}}(p)\neq \cR(p)$.\\
In particular, if $p\in \Sigma\cap \Delta$, then 
\begin{equation}\label{E1-Sec6}
{\mathcal{R}}_{X_{\theta}^{p}}({W^{s}}_{loc,\cR}(q,\Sigma))\cap \cR({W^{s}}_{loc,\cR}(q,\Sigma))=\emptyset,
\end{equation}
\noindent for $q\in \Sigma\cap \Delta$ close to $p$, where $W^{s}_{loc,\cR}(q,\Sigma)$ is the local stable manifold associated to $\cR$.\\
Note that if $\theta=0$, then $X_{\theta}=\phi$.

\indent As $\Delta$ is a compact hyperbolic set, then there are a finite number of point in $\Delta$,  say  $p_1,\dots, p_n$ and  neighborhood $U_i$ of $\phi^{\frac{t_{+}(p_i)}{2}}$, pairwise disjoint as Lemma \ref{Le 6.2},  and such that the projection of $\ds\bigcup_{i}^{n}U_i$ over $\Xi$ along the flow $\phi^{t}$  contains a small Markov partition of $\Delta$.\\
So, we can define the $C^{r}$-vector field $X_{\theta}\in \mathfrak{X}^{r}_{\omega}(M)$ by 

\[ X_{\theta} = \left\{\begin{array}{lll}
X_{\theta}^{p_i} & \mbox{$\text{if} \ \  x\in U_i$};\\
          \ \ & \mbox{ \ } \\
\phi & \mbox{$ \text{otherwise}. $}\end{array} \right.\]
       
\noindent As $\theta$ is small, then the flow of $X_{\theta}$ is still a conservative Anosov flow, since the Anosov flows are robust.\\ Consider now the map $\Phi^{\theta}(x):=\cR^{-1}\circ \cR_{X_{\theta}}(x)$ defined in a small Markov partition of $\Delta$. Then by equation (\ref{E1-Sec6}), the map $\cR_{\theta}:=\cR\circ \Phi^{\theta}$ satisfies 
\begin{equation}\label{E2-Sec6}
{\mathcal{R}}_{\theta}({W^{s}}_{loc,\cR}(q,\Sigma))\cap \cR({W^{s}}_{loc,\cR}(q,\Sigma))=\emptyset, \, \, \text{for any} \, \, \, q\in \Delta.
\end{equation}

\noindent The last equation implies the following Lemma (cf. Appendix \ref{SIRCS} and \cite{MY1}).
\begin{Le}\label{R1-Sec6}
The  family of perturbations ${\mathcal{R}}_{\theta}$ of $\cR$ satisfies that the pair $({\mathcal{R}}_{\theta}, \Delta_{X_{\theta}})$ has the property $V$. Moreover, this property is persistent, \emph{i.e.}, there exists  a $C^{2}$-neighborhood $\mathcal{W}_{\theta}\subset \mathfrak{X}^{r}_{\omega}(M)$ of $X_{\theta}$ such that for all $X\in \mathcal{W}_{\theta}$ the pair $(\cR_{X}, \Delta_X)$ also have the property $V$.  
\end{Le}

\subsection{Description of the set $\mathcal{U}_{X,\Lambda}$}\label{Desc of Functions}
To construct the set  $\mathcal{U}_{X,\Lambda}$ is tied to the "differentiability " of $\textrm{max} F_{\phi}$. But, remember that in general $\textrm{max} F_{\phi}$ is no differentiable (see Remark \ref{RND}).\\
In what follows we give some ``differentiability" to {$\textrm{max} F_{\phi}$ at least for $F\in C^{2}(M,\re)$}, see Lemma \ref{Lmax} below. To achieve  such differentibility, we will need a combinatorial arguments.

\subsubsection{Combinatorial Arguments and differentiability of $\textrm{max} F_{\phi}$}\label{CA}

%
\noindent The following lemma is combinatorial and will be used to show the Lemma \ref{LIC}.
\begin{Le}\label{L5SGF}
Let $A=(a_{ij})_{1\leq i,j\leq n}$ be a matrix such that $a_{ij}\in\{0,1\}$ for any $i, \ j$ and \\$|\{(i,j):a_{ij}=1\}|\geq\frac{99}{100}n^{2}$, then $tr(A^{k})\geq\left(\frac{n}{2}\right)^{k}$ for all $k\geq2$. Moreover, there is a set $Z\subset \{1,2,\dots,n\}$ with $|Z|\geq\frac{4n}{5}$ such that, for any $k\ge 2$ and any $i,j\in Z$, we have
$$(A^{k})_{ij}\geq \frac{4}{5}\left(\frac{3}{5}\right)^{k-2}\cdot n^{k-1}.$$
\end{Le}
\noindent Remember that if $B=(b_{ij})_{1\leq i,j\leq n}$ is a square matrix, then $tr(B)=\sum^{n}_{i=1}b_{ii}$ denotes the trace of $B$.
\begin{proof} 
There is $X\subset\{1,2,\dots,n\}$ with $|X|\geq\frac{9n}{10}$ such that, for any $\ds i\in X$, \\ $|\{j\leq n:a_{ij}=1\}|\geq\frac{9n}{10}$. 
Indeed,  
if there are more than $\frac{n}{10}$ lines in the matrix, each with at least $\frac{n}{10}$ null entries, then the number of null entries of the matrix is greater that $\frac{\ n^2}{100}$, and so $|\{(i,j):a_{ij}=1\}|<n^2-\frac{\ n^2}{100}=\frac{\ 99n}{100}$ which is a contradiction.\\
Analogously, there is $Y\subset\{1,2,\dots,n\}$ with $|Y|\geq\frac{9n}{10}$ such that, for any $\ds j\in Y$,\\ $|\{i\leq n:a_{ij}=1\}|\geq\frac{9n}{10}$.
Let $Z=X\cap Y$, we have $|Z|\geq\frac{9n}{10}+\frac{9n}{10}-n=\frac{4n}{5}$.
If $i,j\in Z$, then $$(A^2)_{ij}=\sum^{n}_{r=1}a_{ir}a_{rj}=\sum_{r\in A_i \cap B_j}a_{ir}a_{rj}=|A_i\cap B_j|\geq\frac{9n}{10}+\frac{9n}{10}-n=\frac{4n}{5},$$
where $A_i=\{j\leq n:a_{ij}=1\}$ and $B_j=\{i\leq n:a_{ij}=1\}$. We will show by induction that if $i,j\in Z$, then 

$$(A^{k})_{ij}\geq \frac{4}{5}\left(\frac{3}{5}\right)^{k-2}\cdot n^{k-1}\ \ \text{for all} \ k\geq 2.$$
In fact, the case $k=2$ was proved above and given $k\ge 2$ for which the statement is true, we have\\
\begin{eqnarray*}
\ds(A^{k+1})_{ij}&=&\sum^{n}_{r=1}(A^{k})_{ir}\cdot a_{rj}\geq \sum_{r\in Z}(A^k)_{ir}\cdot a_{rj}\geq |Z\setminus\{r\in Z:a_{rj}=0\}|\times \frac{4}{5}\left(\frac{3}{5}\right)^{k-2}\cdot n^{k-1}\\
&\geq& \left(\frac{4n}{5}-\frac{n}{10}\right)\frac{4}{5}\left(\frac{3}{5}\right)^{k-2}\cdot n^{k-1}>\frac{4}{5}\left(\frac{3}{5}\right)^{k-1}\cdot n^{k},
\end{eqnarray*}
since $|\{r\in Z:a_{rj}=0\}|\leq\frac{n}{10}$.

Thus, for all $k\geq 2$
$$tr(A^{k})\geq \sum_{i\in Z}(A^{k})_{ii}\geq \frac{4n}{5}\cdot \frac{4}{5}\left(\frac{3}{5}\right)^{k-2}\cdot n^{k-1}>\left(\frac{3}{5}\right)^{k}\cdot n^k>\left(\frac{n}{2}\right)^{k}.$$

\end{proof}
\begin{R}\label{R-HD}
Suppose that the matrix $A$ as in \emph{Lemma \ref{L5SGF}}, is the matrix of transitions for a regular Cantor set $K$ with Markov partition $R=\{R_1, R_2,\cdots, R_n\}$ defined by an expansive map $\psi$ satisfying $C^{-1}/\varepsilon<|\psi'(x)|<C/\varepsilon, \forall x\in \cup_{i\le n}R_i$, for a suitable constant $C$ $($with $\log C \ll \log \varepsilon^{-1})$. From \emph{Lemma \ref{L5SGF}}, we get a set $Z$ of indices with $|Z|\geq \frac{4n}{5}$.
Fix indices $\tilde{i}, \tilde{j}\in Z$ such that $a_{\tilde{i}\tilde{j}}=1$. Consider a Markov partition for $\psi^{k+2}$ corresponding to the words in the set $$X=\{\tilde{j}r_1r_2\cdots r_k\tilde{i}:r_i\leq n \ \ \text{and} \ \ a_{\tilde{j}r_1}=a_{r_1r_2}=\dots=a_{r_{k-1}r_k}=a_{r_k\tilde{i}}=1\}.$$ By \emph{Lemma \ref{L5SGF}}, $|X|=(A^{k+1})_{\tilde{i}\tilde{j}}\geq \frac{4}{5}\left(\frac{3}{5}\right)^{k-1}\cdot n^{k}>\left(\frac{n}{2}\right)^{k}$, since $a_{\tilde{i}\tilde{j}}=1$, any transition between two words in $X$ is admissible.\\
\noindent Consider the regular Cantor set 
$$\tilde{K}:=\{\alpha_1\alpha_2\alpha_3\dots|\alpha_i\in X, \forall i\ge 1\}\subset K.$$
Taking $k$ large enough, since  $|(\psi^{k+2})'|<\left(\frac{C}{\varepsilon}\right)^{k+2}$, then 
\begin{eqnarray*}
HD(\tilde{K})&>&\frac{\log \left(\frac{n}{2}\right)^{k}}{\log\left(\frac{C}{\varepsilon}\right)^{k+2}}=\frac{k}{k+2}\cdot\dfrac{\log n-\log2}{\log C-\log\varepsilon}=(1-o(1))\dfrac{\log n}{\log(\varepsilon^{-1})}\\
&=&(1-o(1))\dfrac{\log n}{\log(C^{-1}/\varepsilon)}\ge (1-o(1))HD(K)\\
\end{eqnarray*}

\noindent It follows that $HD(\tilde{K})\sim HD(K)\sim \dfrac{\log n}{\log(\varepsilon^{-1})}$.
\end{R}
We use the above Remark to understand  the behavior of the horseshoe $\Delta$ when it is intersected by a finite number of $C^{1}$-curves.\\
\begin{Le}{\emph{\textbf{Intersection of curves with $\Delta$}}}\label{LIC}\\
Let $\alpha=\{\alpha_i:[0,1]\to \Xi, i\in\{1,\dots,m\}\}$ be a finite family of $C^{1}$-curves. Then for all $\epsilon>0$ there are sub-horseshoes $\Delta_{\alpha}^{s}$,  $\Delta_{\alpha}^{u}$ of $\Delta$ such that $\Delta_{\alpha}^{s,u} \cap \alpha_i([0,1])=\emptyset$ for any $i\in\{1,\dots,m\}$ and 
$$HD(K_{\alpha}^{s})\geq HD(K^s)-\epsilon \ \text{and} \ \ HD(K_{\alpha}^{u})\geq HD(K^u)-\epsilon,$$
\noindent where $K_{\alpha}^{s}$, $K^{s}$ are regular Cantor sets that describe the geometry transverse of the unstable foliation $W^{u}(\Delta_{\alpha}^s), \  W^{u}(\Delta)$ respectively, and $K_{\alpha}^{u}$, $K^{u}$ are regular Cantor set that describe the geometry transverse of the stable foliation $W^{s}(\Delta_{\alpha}^u), \  W^{s}(\Delta)$, respectively \emph{(cf. Appendix \ref{sec EMAH})}.
 \end{Le}

\noindent Before starting the proof of the previous lemma, we introduce some definitions and remarks.\\

\indent Let us fix a Markov partition $R$ of $\Delta$.  Given $R(\ua)\in R$, for an admissible word $\ua=(a_{i_1},\cdots,a_{i_r})$. We denote $\left|(a_{i_1},\cdots,a_{i_r})\right|$ the diameter of the projection on $W^{s}_{loc}$ of $R(\ua)$ along to the  foliation $\mathcal{F}^{u}$ (cf. the construction of $K^{s}$ in  Appendix \ref{sec EMAH}).
Fix $a_r, \ a_s$ such that the pair $(a_r,a_s)$ is admissible. Given $\epsilon>0$, we have the following definition.
\begin{Defi}\label{D2SGF}
A piece $(a_{i_1},\cdots,a_{i_k})$ $($in the construction of $K^{s}$\emph{)} is called an \emph{$\epsilon$-piece} if 
$$ |(a_{i_1},\cdots,a_{i_k})|<\epsilon \ \text{and} \ |(a_{i_1},\cdots,a_{i_{k-1}})|\geq \epsilon.$$
\end{Defi}

\noindent Put  $$X_{\epsilon}=\{ \epsilon \text{-piece} \ (a_{i_1},\cdots,a_{i_k}):i_1=s \ \text{and} \ i_k=r\}=\{\theta_{1},\dots,\theta_{N}\}.$$

\noindent Notice that $\theta_{i}\theta_{j}$ is an admissible word for every $i, j\le N$. We define  the Cantor set 
$$K(X_{\epsilon}):=\{\theta_{j_{1}}\theta_{j_{2}}\cdots\theta_{j_k}\cdots | \theta_{j_{i}}\in X_{\epsilon},\forall i\ge 1\}\subset K^{s}.$$

\indent Notice that $N\sim \epsilon^{-d_s}$, where $d_s=HD(K^s)$, and so $HD(K(X_\epsilon))$ is close to $HD(K^{s})$ provided $\epsilon$ is small enough.\\

\indent Dividing the curves in smaller curves if necessary, we can assume that the finite family $\alpha$ is formed by curves that are graphs of $C^{1}$-functions of $W^{s}(\Delta)$ on $W^{u}(\Delta)$ or from $W^{u}(\Delta)$ on $W^{s}(\Delta)$.\\
\indent Denote by $I_{\theta_i}$ the interval associated with $\theta_i$ in the construction of $K^{s}$. There is a constant $C>1$ (which depends on the geometry of the horseshoe $\Delta$, but not on $\epsilon$) such that
$$C^{-1}\epsilon<|I_{\theta_{i}}|<C\epsilon.$$

\noindent For each $I_{\theta_i}$, with $\theta_i=(a_{i_1},\cdots,a_{i_k})$, we associate the interval $I'_{\theta_{i}^{t}}$ corresponding to the transposed sequence $\theta_{i}^{t}=(a_{i_k},\cdots,a_{i_1})$ in the construction of $K^{u}$ (unstable Cantor set),  by an abuse of language, we will say that the interval $I'_{\theta_{i}^{t}}$ is the ``\emph{transposed}'' interval of $I_{\theta_i}$ (and vice-versa). Then, since $\Delta$ is horseshoe there exists $\beta\ge 1$ (which depends on the geometry of the horseshoe $\Delta$ but not on $\epsilon$ or $k$) such that
$$C^{-1}|I_{\theta_i}|^{\beta}<|I'_{\theta_{i}^{t}}|<C|I_{\theta_{i}}|^{1/\beta}.$$
\begin{R}\label{conservative geometry}
In the conservative case, \emph{i.e.}, when the horseshoe is defined by a diffeomorphism that preserves a smooth measure, the above inequality holds with $\beta=1$. 
\end{R}

\begin{proof} [\bf{Proof of Lemma \ref{LIC}}]
We prove the stable case since the unstable case is analogous. For this sake, we consider the related position of the family of curves $\alpha$ with respect to the stable and unstable manifolds.
\item- \underline{First case}. (Graph of a $C^{1}$-function from $W^{s}(\Delta)$ on $W^{u}(\Delta)$). In this case, consider the image $P$ of $I_{\theta_i}$ by this function. Then, $C$ and $\epsilon$ (of the above discussion) can be taken such that $|P|\leq C^2\epsilon$. Let $P'$, the smallest interval of the construction of $K^{u}$ containing $P$. Then, if $J\in W^{s}(\Delta)$ is the transposed interval of $P'$, we have $|J|\leq (C^2\epsilon)^{1/\beta}$. Then 
$$\#\{I_{\theta_{j}}:I_{\theta_j}\cap J\neq \emptyset\}\leq C\left(\frac{(C^2\epsilon)^{1/\beta}}{\epsilon}\right)^{d_s}=\tilde{C}\epsilon^{d_s(1/\beta-1)},$$
where $d_s$ is the Hausdorff dimension of stable Cantor set.\\
Thus, $$\#\{(I_{\theta_i}, I'_{\theta_{j}^{t}}):I_{\theta_i}\times I'_{\theta_{j}^{t}}\ \text{intersects the curve} \}\leq \epsilon^{-d_s}\tilde{C}\epsilon^{d_s(1/\beta-1)}=\tilde{C}\epsilon^{d_s(1/\beta-2)}\ll\epsilon^{-2d_s}.$$

\item- \underline{Second case}. (Graph of a $C^{1}$-function from $W^{u}(\Delta)$ on $W^{s}(\Delta)$). In this case, consider the image $J'$ of $I'_{\theta_{i}^{t}}$. Then, $|J'| \leq c |I'_{\theta_{i}^{t}}|\leq c(C\epsilon)^{1/\beta}$, ($J'$ is the image of $I'_{\theta_{i}^{t}}$ by a $C^{1}$-function), so we have, analogously,

$$\#\{I_{\theta_{i}}:I_{\theta_i}\cap J'\neq\emptyset\}\leq {\hat{C}}\epsilon^{d_s(1/\beta-1)}$$
and
$$\#\{(I_{\theta_i}, I'_{\theta_{j}^{t}}):I_{\theta_i}\times I'_{\theta_{j}^{t}}\ \text{intersects the curve}\}\leq \epsilon^{-d_s}\hat{C}\epsilon^{d_s(1/\beta-1)}=\hat{C}\epsilon^{d_s(1/\beta-2)}\ll\epsilon^{-2d_s}.$$
Note that $\epsilon^{-2d_s}\sim N^{2}=$ the total number of transitions $\theta_i\theta_j$.\\
\noindent We say that $\theta_U\theta_V$ is a prohibited transition if and only if some curve of the family $\alpha$ intersects the rectangle $I_{\theta_{U}}\times I'_{\theta_{V}^{t}}.$\\

\indent Consider the admissible word $\theta_i\theta_j\theta_k\theta_s$ with $\theta_{i},\theta_{j},\theta_{k}, \theta_{s}\in X_{\epsilon}$. This word generates an interval of the size of the order of $\epsilon^4$ in the construction of $K^{s}$.\\

\noindent We say that $\theta_i\theta_j\theta_k\theta_s$ is a prohibited word, if within there is a prohibited transition $\theta_{U}\theta_{V}$

$$\overbrace{---}^{\theta_i}\overbrace{---}^{\theta_j}\overbrace{---}^{\theta_k}\overbrace{---}^{\theta_s}$$
$$\underbrace{--}_{\rho}\underbrace{------}_{\theta_{U}\theta_{V}}\underbrace{----}_{\beta}.$$

\noindent Denote by $PW$ the set of the prohibited words $\theta_i\theta_j\theta_k\theta_s$. We want to now estimate $|PW|$.\\
In fact: $|I_{\rho}||I_{\beta}|\sim \epsilon^2\sim 2^{-2n}$, then there is $t\leq 2n$ such that $|I_{\rho}|\sim 2^{-t}$ and $|I_{\beta}|\sim 2^{t-2n}$.\\
Thus, $\#\{I_{\rho}\}\sim (2^{-t})^{-d_s}=2^{td_s}$ and $\#\{I_{\beta}\}\sim (2^{-(2n-t)})^{-d_s}=2^{(2n-t)d_s}$. Therefore, for some constant $\tilde C>1$ (as in the first part of the proof), we have that\\
$$|PW|\leq \tilde C \cdot (2n)\cdot 2^{td_s}2^{(2n-t)d_s}\epsilon^{d_s(1/\beta-2)}\leq 2\tilde C\log \epsilon^{-1}\epsilon^{d_s(1/\beta-4)}\ll \epsilon^{-4d_s}$$
\noindent the last inequality follows from $2\tilde C(\log\epsilon^{-1)}\epsilon^{d_s/\beta}\ll 1$.
\ \\

\noindent Then, the total of prohibited words $\theta_i\theta_j\theta_k\theta_s$ is much less than $\epsilon^{-4d_s}\sim N^4$, the total number of words $\theta_i\theta_j\theta_k\theta_s$.

\noindent Consider $\ds A=(a_{(i,j)(k,s)})$ for $(i,j),(k,s)\in\{1,\dots,N\}^2$ the matrix defined by 
\[ a_{(i,j)(k,s)} = \left\{ \begin{array}{lll}
         1 & \mbox{if $\theta_{i}\theta_{j}\theta_{k}\theta_{s} \ \text{is not prohibited}$};\\
          & \\ 
       0 & \mbox{if $\theta_i\theta_j\theta_{k}\theta_{s}$ {is prohibited for some} \ 
       $\theta_U\theta_{V}$}.\end{array} \right. \] 
Put $\tilde{\theta}_{ij}=\theta_i\theta_j$ for $i,j\le N$. Define  $\widetilde{K}$ the regular Cantor set 
$$\widetilde{K}:=\{\tilde{\theta}_{i_1j_1}\tilde{\theta}_{i_2j_2}\cdots\tilde{\theta}_{i_nj_n}\cdots | a_{(i_k,j_k)(i_{k+1},j_{k+1})}=1,\forall k\ge 1\}\subset K^{s}.$$       
By the previous discussion, we have $\#\{a_{(i,j)(k,s)}:a_{(i,j)(k,s)}=1\}\geq\frac{99}{100}({N^2})^{2}$, so by the
Remark \ref{R-HD} we have $HD(\widetilde{K})\sim HD(K(X_{\epsilon}))\sim HD(K^s)$. Consider the sub-horseshoe of $\Delta$ defined by 

$$\Delta^{s}_{\alpha}:=\bigcap_{n\in\mathbb{Z}}\cR^{n}\left(\bigcup_{(i,j),(k,s)\in \{1,2,\dots,N\}^2, \  a_{(i,j)(k,s)}=1} (R(\tilde\theta_{ij})\cap {\cR}^{-1}(R(\tilde\theta_{ks}))\right),$$ 
where $R(\tilde\theta_{ij})$ is the rectangle associated to the word $\tilde\theta_{ij}$.\\
Then, the stable regular Cantor set $K^s_{\alpha}$ describing the transverse geometry of the unstable foliation $W^{u}(\Delta^{s}_{\alpha})$ is equal to $\widetilde{K}$ . Then by the above discussion we have that 

$$HD(K^s_{\alpha})\sim HD(K^s)$$
and by definition of $\Delta^{s}_{\alpha}$ we have that $\Delta^{s}_{\alpha}\cap \alpha_i=\emptyset, \forall i\le m$. This concludes the proof.
\end{proof}
Now we can prove some lemmas, which give some differentiability to $\textrm{max}F_{\phi}$. To get there, we introduce some subsets of $C^2(M,\re)$. First, we consider the family of one parameter  $\beta>0$, $\mathcal{B}_{\phi,\beta}^{s(u)}$ defined as follows.
\begin{Defi}\label{DLmax} We  say that $F\in \mathcal{B}_{\phi, \beta}^{s(u)}\subset C^2(M,\re)$ whenever 
\begin{enumerate}
\item[\emph{(i)}] There exists a sub-horseshoe $\Delta_{F}^{s(u)}$ of $\Delta$ with $HD(K_{F}^{s(u)})> HD(K^{s(u)})-2\beta$,
\item[\emph{(ii)}] There exists a Markov partition $R_{F}^{s(u)}$ of $\Delta_{F}^{s(u)}$, respectively, such that the function $\textrm{max} F_{\phi}|_{\Xi\cap R_{F}^{s(u)}}\in C^{1}(\Xi\cap R_{F}^{s(u)},\re)$,
\end{enumerate}
where $K_{F}^{s(u)}$, $K^{s(u)}$ are the stable (unstable) Cantor sets associated to $\Delta_{F}^{s(u)}$ and $\Delta$, respectively. 
\end{Defi}

\begin{Le}\label{Lmax}
For any $\beta>0$ small enough, the sets $\mathcal{B}_{\phi,\beta}^{s(u)}$ are dense and $C^{2}$-open sets. 
\end{Le}
Before we present the proof of Lemma \ref{Lmax}, let us to consider an auxiliary set $\mathcal{N}_{\phi}$ of functions defined as follows. 

Once again we cover $\Lambda$ with a finite number of tubular neighborhoods $U_k$, $1\leq k\leq m$ whose boundaries are $C^\infty$ - good cross sections pairwise disjoints (as Remark \ref{R11}) with $\Xi=\bigcup\limits_{i=1}^k\Sigma_i$. \\ For each $k$, let us fix coordinates $(x_1(k),x_2(k),x_3(k))$ on $U_k$ such that $x_3(k)$ is the flow direction and $U_s\cap\Xi=\{x_3(s)=0\}\cup\{x_3(s)=1\}$. 

\begin{Defi}\label{d.N}
We say that $F\in\mathcal{N}_{\phi}\subset C^r(M,\re)$, $r\geq 4$,  whenever: 
\begin{itemize}
\item[\emph{(i)}] $0$ is a regular value of the restriction of $\frac{\partial F}{\partial x_3(k)}$ to $U_k\cap\Xi$; 
\item[\emph{(ii)}] $0$ is a regular value of $\frac{\partial^3 F}{\partial x_3(k)^3}$; 
\item[\emph{(iii)}] $0$ is a regular value of the functions $\frac{\partial^2 F}{\partial x_3(k)^2}$ and $\frac{\partial^2 F}{\partial x_3(k)^2}|_{\{\frac{\partial^3 F}{\partial x_3(s)^3}=0\}}$; 
\item[\emph{(iv)}] $0$ is a regular value of the functions $\frac{\partial F}{\partial x_3(k)}|_{\{\frac{\partial^2 F}{\partial x_3(k)^2}=0\}}$ and $\frac{\partial F}{\partial x_3(k)}|_{\{\frac{\partial^3 F}{\partial x_3(k)^3}=0\}\cap \{\frac{\partial^2 F}{\partial x_3(k)^2}=0\}}$,
\end{itemize} 
for each $1\leq k\leq m$. 
\end{Defi}

\begin{Le}\label{l.N} The set $\mathcal{N}_{\phi}$ is dense in $C^r(M,\re)$,  $r\geq 4$. 
\end{Le}

\begin{proof} Given a function $F\in C^r(M,\re)$,  $r\geq 4$, let us consider the three-parameter family 
$$F_{a,b,c}(x_1,x_2,x_3) = F(x_1,x_2,x_3) - c x_3^3/6 - b x_3^2/2 -a x_3$$
where $a,b,c\in\mathbb{R}$. 

By Sard's theorem, we can fix first a very small regular value $c\approx 0$ (close enough to $0$) of $\frac{\partial^3 F}{\partial x_3^3}$, then a very small regular value $b\approx 0$ of both $\frac{\partial^2 F}{\partial x_3^2} - c x_3$ and its restriction to $\{\frac{\partial^3 F}{\partial x_3^3}=c\}$, and finally a very small regular value $a\approx 0$ of $(\frac{\partial F}{\partial x_3} - c x_3^2/2 - b x_3)|_{\{\frac{\partial^2 F}{\partial x_3^2} - cx_3 = b\}}$, $(\frac{\partial F}{\partial x_3} - c x_3^2/2 - b x_3)|_{\{\frac{\partial^3 F}{\partial x_3^3} = c\}\cap\{\frac{\partial^2 F}{\partial x_3^2} - cx_3 = b\}}$ and $(\frac{\partial F}{\partial x_3} - c x_3^2/2 - b x_3)|_{\{x_3 = 0\}\cup\{ x_3 = 1\}}$. 

For a choice of parameters $(a,b,c)$ as above, we have that $F_{a,b,c}\in\mathcal{N}_{\phi}$: indeed, this happens because $\frac{\partial^3 F_{a,b,c}}{\partial x_3^3} = \frac{\partial^3 F}{\partial x_3^3} - c$, $\frac{\partial^2 F_{a,b,c}}{\partial x_3^2} = \frac{\partial^2 F}{\partial x_3^2} - c x_3-b$ and $\frac{\partial F_{a,b,c}}{\partial x_3} = \frac{\partial F}{\partial x_3} - c x_3^2/2 - b x_3-a$. Clearly, $F_{a,b,c}$ is arbitrarily close to $F$, which  proves the lemma. 
\end{proof}
\ \\
By definition, if $F\in\mathcal{N}_{\phi}$, then $\mu_k:=\{\frac{\partial F}{\partial x_3(k)}=0\}\cap (U_k\cap \Xi)$ is a curve (by (i)), and $J_k:=\{\frac{\partial F}{\partial x_3(k)}=0\}\cap \{\frac{\partial^2 F}{\partial x_3(k)^2}=0\}$ is a curve intersecting the surface $\{\frac{\partial^3 F}{\partial x_3(k)^3}=0\}$ at a finite set $\Pi_k$ of points (by (ii), (iii) and (iv)). 

Note that if $(x_1,x_2,0), (x_1,x_2,1)\notin \mu_k$ (\emph{i.e.}, the orbit is transverse to the cross-sections) and the piece of orbit $(x_1,x_2,z)$, $0\leq z\leq 1$, does not intersect $J_k$. Then there is a neighborhood $V$ of $(x_1, x_2, 0)\in U_k\cap\Xi$ and a finite collection of disjoint graphs $\{(x,y,\psi_j(x,y)): (x,y,0)\in V\}$, $1\leq j\leq n$ such that if $f(x_1',x_2')=\textrm{max} F_{\phi}(x_1',x_2')=F(x_1',x_2',t')$ with $(x_1',x_2',0)\in V$, then $t'=\psi_j(x_1',x_2')$ for some $j$. 

\begin{proof}[\emph{\textbf{Proof of Lemma \ref{Lmax}}}]
The openness of $\mathcal{B}_{\phi,\beta}^{s(u)}$ is a consequence of its own definition. Indeed, given $F\in\mathcal{B}_{\phi,\beta}^{s(u)}$, then for any $G\in C^{2}(M,\re)$ sufficiently close to $F$, we have that $\textrm{max} G_{\phi}|_{\Xi\cap R_F^{s(u)}}\in C^{1}(\Xi\cap R_{F}^{s(u)}, \re)$, therefore, we can taken $R_{G}^{s(u)}=R_{F}^{s(u)}$ and $\Delta_{G}^{s(u)}=\Delta_{F}^{s(u)}$, and this concludes the proof of openness. \\
\indent Let us to make the proof of density for the stable case. The unstable case is analogue. \\ 
By Lemma \ref{l.N}, it is sufficient to prove that $\mathcal{B}_{\phi,\beta}^{s}\cap C^r(M,\re)$, $r\geq 4$, is dense in $\mathcal{N}_{\phi}$. Observe that, in the  statements of the proof of Lemma \ref{l.N}, we consider $F\in\mathcal{N}_{\phi}$ as above. Our discussion so far says that the curves $\mu_k$ and the projections of the curves $J_k$ in the flow direction ($x_3$-coordinate) is a  finite union $J$ of $C^1$ curves contained in $\Xi$ such that, for each $y\in D_\mathcal{R}\setminus J$. The value $\textrm{max} F_{\phi}(z)$ for $z$ near $y$ is described by the values of $\textrm{max} F_{\phi}$ at a finite collection of graphs transverse to the flow direction. \\
\indent From the Lemma \ref{LIC}, given $\beta>0$ small there is a sub-horseshoe $\Delta_{J}$ such that 
$$HD(K_{J}^{s})\geq HD(K^{s})-\beta  \ \  \text{and} \ \  \Delta_{J}\cap \alpha=\emptyset,$$
for each curve $\alpha\in J$.
In other terms, using the notation in the paragraph after proof of Lemma \ref{l.N}, our task is reduced to perturb $F$ in such a way that $f(x_1',x_2')$ are given by the values of $F$ on an \emph{unique} graph $(x_1',x_2',\psi(x_1',x_2'))$. \\
In this direction, let $V$ be a small neighborhood of $\Delta_J$ such that $V\cap \alpha=\emptyset$ for every $\alpha\in J$. Note that the value of $F$ at any point $(x,y)\in V$ is described by finitely many disjoint  graphs $\psi_j$, $1\leq j\leq k$. \\
\indent 
Let $g_{1j}(x_1,x_2)=F(x_1,x_2,\psi_{1}(x_1,x_2))-F(x_1,x_2,\psi_{j}(x_1,x_2))$ for $j\neq 1$ and consider $\gamma_1>0$ small regular value of $g_{1j}$ for all $j\neq 1$. Take $\xi_1$ a $C^{\infty}$-function close to the constant function $0$ and equal to $-\gamma_1$ in neighborhood of $\{z=\psi_{1}(x_1,x_2)\}$ and $0$ outside. So, the function $F+\xi_1$ is close to $F$. Now we define the function $$g_{1j}^{\gamma_1}(x_1,x_2)=(F+\xi_1)(x_1,x_2,\psi_{1}(x_1,x_2))-(F+\xi_1)(x_1,x_2,\psi_{j}(x_1,x_2))=g_{1j}(x_1,x_2)-\gamma_1.$$ 
Put $F_1:=F+\xi_{1}$ and define $g_{2j}(x_1,x_2)=F_1(x_1,x_1,\psi_2(x_1,x_2))-F_{1}(x_1,x_2,\psi_{j}(x_1,x_2))$ for $j\neq 2$ and let $\gamma_2>0$ small regular value of $g_{2j}$ for all $j\neq2$. Take $\xi_2$ a $C^{\infty}$-function close to the constant function $0$ and equal to $-\gamma_2$ in neighborhood of $\{z=\psi_{2}(x_1,x_2)\}$ and $0$ outside. So, the function $F_{1}+\xi_2$ is close to $F$ and again, define the function $$g_{2j}^{\gamma_2}(x_1,x_2)=(F_{1}+\xi_2)(x_1,x_2,\psi_{2}(x_1,x_2))-(F_{1}+\xi_2)(x_1,x_2,\psi_{j}(x_1,x_2))=g_{2j}(x_1,x_2)-\gamma_2.$$
Inductively, define $F_{s-1}=F_{s-2}+\xi_{s-1}$ and $$g_{sj}(x_1,x_2)=F_{s-1}(x_1,x_2,\psi_{s}(x_1,x_2))-F_{s-1}(x_1,x_2,\psi_{j}(x_1,x_2))$$
for $j\neq s$. Let $\gamma_s>0$ small regular value of $g_{sj}$ for all $j\neq s$. Take $\xi_s$ a $C^{\infty}$-function close to the constant function $0$ and equal to $-\gamma_s$ in neighborhood of $\{z=\psi_{s}(x_1,x_2)\}$ and $0$ outside. So, the function $F_s:=F_{s-1}+\xi_s$ is close to $F$ and $$g_{sj}^{\gamma_s}(x_1,x_2):=F_{s}(x_1,x_2,\psi_{s}(x_1,x_2))-F_{s}(x_1,x_2,\psi_{j}(x_1,x_2))=g_{sj}(x_1,x_2)-\gamma_s.$$ 

\noindent Therefore, for each $s=1,\dots,k-1$,  $\Gamma_{s}:=\ds\bigcup_{j\neq s}(g_{sj}^{\gamma_s})^{-1}(0)$ is a finite collection of $C^1$ curves in $\Xi$, $1\leq s\leq k-1$. Put $\Gamma:=\ds\bigcup_{s=1}^{k-1}\{\Gamma_{s}\}$, then by Lemma \ref{LIC} there is a sub-horseshoe $\Delta_{\Gamma}$ of $\Delta_{\alpha}$ such that 
\begin{equation}\label{ED}
HD(K_{\Gamma}^{s})\geq HD(K_{\delta}^{s})-\beta \geq HD(K^{s})-2\beta \ \ \text{and} \ \ \Delta_{\Gamma}\cap \gamma=\emptyset,
\end{equation}
for each $\gamma\in \Gamma$.\\
\indent To finish the proof, consider the perturbation  $F+\xi^{k}$ of $F$, where  $\xi^{k}:=\xi_1+\cdots+\xi_{k-1}$, then if $l<j$, we have 
\begin{eqnarray*}
(F+\xi^{k})(x_1,x_2,\psi_{j}(x_1,x_2))&-&(F+\xi^{k})(x_1,x_2,\psi_{l}(x_1,x_2))=\\
(F+\xi_1+\cdots+\xi_{j})(x_1,x_2,\psi_{j}(x_1,x_2))&-&(F+\xi_1+\cdots+\xi_{l})(x_1,x_2,\psi_{l}(x_1,x_2))=\\
(F+\xi_1+\cdots+\xi_{j})(x_1,x_2,\psi_{j}(x_1,x_2))&-&(F+\xi_1+\cdots+\xi_{l}+\cdots+\xi_j)(x_1,x_2,\psi_{l}(x_1,x_2))\\
&=& g_{jl}^{\gamma_l}(x_1,x_2).
\end{eqnarray*}
Thus, if $(x_1,x_2) \in \Delta_{\Gamma}$, then 
\begin{equation}\label{ED1}
(F+\xi^{k})(x_1,x_2,\psi_{j}(x_1,x_2))\neq (F+\xi^{k})(x_1,x_2,\psi_{s}(x_1,x_2)) \ \ \text{for all} \ \ j\neq s.
\end{equation}
So, taking a Markov partition $R_{\Gamma}$ of $\Delta_{\Gamma}$ with a diameter  small enough, for each $y\in R_{\Gamma}$ the values of $\textrm{max} (F+\xi^{k})_{\phi}$ near $y$ are described by the values of $F+\xi^{k}$ at a unique graph. Hence, for each $y\in R_{\Gamma}$, one has that $\textrm{max} (F+\xi^{k})_{\phi}(y)=F(\phi^{t(y)}(y))$ for a unique $0\leq t(y)\leq t_+(y)$ depending in a  $C^1$ way on $y$. Therefore, we conclude that  $\textrm{max} (F+\xi^{k})_{\phi}|_{\Xi\cap R_{\Gamma}}$ is a $C^{1}$-function. Therefore, the function $F+\xi^{k}\in (\mathcal{B}_{\phi,\beta}^{s}\cap C^r(M,\re))$, $r\geq 4$, which concludes the proof of the lemma.
\end{proof}

\noindent Keeping the notation of the previous Lemma we have:
\begin{R}\label{Cmax}
The definition of $\mathcal{B}_{\phi,\beta}^{s(u)}$ depends on the vector field $\phi$. If $X$ is a vector field $C^{2}$-sufficiently close to $\phi$, then $\mathcal{B}_{\phi,\beta}^{s(u)}=\mathcal{B}_{X,\beta}^{s,u}$.  
\end{R}
\subsubsection{The set ${\mathcal{U}}_{X,\Lambda}$}\label{Descript of function} 
Given a compact hyperbolic set $K$ for $\cR$ and a Markov partition $R$ of $K$, we define the set 
\begin{equation}\label{E-Descrip}
{H}_{1}(\cR,K)=\left\{f\in C^{1}(\Xi\cap R,\re):\#M_{f}(K)=1, \  z\in M_{f}(K), \ D\cR_{z}(e_{z}^{s,u})\neq 0\right\},
\end{equation}
where $M_{f}(K):=\{z\in K: f(z)\geq f(x)\ \text{for all} \ x\in K\}$, the set of maximum points of $f$ on $K$ and $e_{z}^{s,u}$ are unit vectors in $E^{s,u}_{\Xi}(z)$, respectively (cf. \cite[section 3]{RM}).\\
\indent Note also, by Remark \ref{R_pert}, that for any $X \in \mathfrak{X}^{2}_{\omega}(M)$ sufficiently close of $\phi$, we have that $HD(\Delta_{X})>1$. Thus, we have\begin{Defi}
We say that  $F\in\mathcal{U}_{X,\Lambda}\subset C^2(M,\re)$, whenever 
\begin{itemize}
\item[\emph{(i)}] There exists a sub-horseshoe $\Delta_F$ of $\Delta_X$ with $HD(\Delta_F)>1$ and neighborhood $R_F$ of $\Delta_F$ such that 
$$\textrm{max}F_{X}|_{\Xi\cap R_F}\in C^1(\Xi\cap R_F,\re).$$
\item[\emph{(ii)}] $\textrm{max}F_{X}\in H_1(\cR_{X},\Delta_F)\subset C^{1}(\Xi\cap R_F,\re)$. 
\end{itemize}
\end{Defi}
\begin{Le}\label{set of functions}
The set ${\mathcal{U}}_{X,\Lambda}$ is dense and $C^2$-open set. 
\end{Le}
\begin{proof}
By definition the set ${\mathcal{U}}_{X,\Lambda}$ is  open. By Lemma \ref{Lmax} our task is simply to prove that ${\mathcal{U}}_{X,\Lambda}$ is dense in ${\mathcal{B}}_{X,\beta}^{s}\cup {\mathcal{B}}_{X,\beta}^{u}$ for some $\beta$ small enough. Indeed, fix $\beta>0$ small enough such that $(HD(\Delta_X)-4\beta)>1$  and let $F\in {\mathcal{B}}_{X,\beta}^{s}\cup {\mathcal{B}}_{X,\beta}^{u}$, then by Lemma \ref{Lmax}, consider the sub-horseshoe $\Delta_F=\Delta_{F}^{s}\cup \Delta_{F}^{u}$ and $R_F=R_F^s\cup R_F^{u}$, therefore by definition of ${\mathcal{B}}_{X,\beta}^{s}\cup {\mathcal{B}}_{X,\beta}^{u}$ we can conclude that 
$$ HD(K_F^s)+HD(K_F^u)\geq HD(K^s_{X})+HD(K^u_{X})-4\beta.$$
Thus $$HD(\Delta_F)\geq HD(\Delta_X)-4\beta>1,$$
since  $HD(\Delta_X)=HD(K^s_{X})+HD(K^u_{X})$ (cf. \cite{PT}) and Remark \ref{R_pert}. 
To conclude the proof, we need some appropriate perturbation of $F$ to become   $\textrm{max}F_{X}$ an element of $ H_1(\cR_{X},\Delta_F)$. Consider a point $x\in \Delta_F$. Recall that, in a small neighborhood of $x$, the values of $\textrm{max}F_{X}$ are given by the values of $F$ on a graph $(x_1 , x_2 , \psi(x_1 , x_2 ))$. Now we can employ the argument of Section 3 in \cite{RM}  to
find arbitrarily small function $g(x_1,x_2)$ such that the functions $F_g (x_1 , x_2 , t) :=
F (x_1 , x_2 , t) +g(x_1,x_2)$ near the graph $(x_1 , x_2 ,\psi(x_1 , x_2 ))$ (and coinciding with $F$
elsewhere) with the property that  $\textrm{max}(F_g)_{X}$ is an element of $ H_1(\cR_{X},\Delta_F)$, as we wished. 
\end{proof}
 
\begin{proof}[\bf Proof of  Theorem \ref{Theorem 2}]
We consider a family of perturbations $X_{\theta}$ of $\phi$ as Subsection \ref{Family of Pert}. Note also, by Remark \ref{R_pert} $HD(\Delta_{X_\theta})>1$. Let $F\in\mathcal{U}_{X_{\theta},\Lambda}$, then 
there exists a sub-horseshoe $\Delta_F$ of $\Delta_{X_\theta}$ with $HD(\Delta_F)>1$ and neighborhood $R_F$ of $\Delta_F$ such that 
$\textrm{max}F_{X_\theta}|_{\Xi\cap R_F}\in H_1(\cR_{X_{\theta}},\Delta_F)\subset C^{1}(\Xi\cap R_F,\re)$. The Lemma \ref{R1-Sec6} provides that the pair $(X_{\theta}, \cR_{\theta})$ has the property $V$, then by Main Theorem of \cite{RM} we can conclude 

$$\text{int}\, M(\cR_{\theta}, \Delta_{X_{\theta}},\textrm{max} F_{X_\theta}|_{\Xi\cap R_F})\neq \emptyset \ \ \text{and} \ \ \text{int}\, L(\cR_{\theta}, \Delta_{X_{\theta}},\textrm{max} F_{X_\theta}|_{\Xi\cap R_F})\neq \emptyset.$$
The property of persistence  is also a consequence of Lemma \ref{R1-Sec6} and the Main Theorem of \cite{RM}.
This completes the proof of Theorem \ref{Theorem 2} (and,  \emph{a fortiori},  Theorem \ref{Theorem 1}).
\end{proof}

\subsection{Anosov Geodesic flow and Anosov suspension flow}\label{Sec 4.3}
In this section,  we prove Corollary \ref{C1-Theorem 1} and Corollary \ref{C2-Theorem 1}  using  Theorem \ref{Theorem 1}.
\subsubsection{Proof of Corollary \ref{C1-Theorem 1}}\label{PC1}
We can note that when the manifold $M$ is the unitary tangent  bundle of a complete Riemannian manifold $N$ endowed with a metric $g_0$ of negative pinched curvature, then the geodesic flow, $\phi^t_{_{0}}$, on $SN$ is Anosov. In order to use  Theorem \ref{Theorem 1}, we need to construct a basic set for $\phi^{{t}}_{_{0}}$ with Hausdorff dimension greater than $2$. For this sake we used the following theorem:\\



\noindent{\textbf{Theorem (\cite{D1} and \cite{D})}\label{T1GCS}
 \emph{Let $N$ be a complete Riemannian manifold of finite volume of dimension $n$, such that all the sectional curvatures are bounded between two negative constants. Let $C$ be the set of points in $SN$ whose orbit through of the geodesic flow is bounded. Then the Hausdorff dimension of $C$, $HD(C)$, is equal to $2n-1$.}\\
 
As a corollary of the above theorem we have:

\begin{Le}\label{Dodson}
In the same conditions of the last theorem, there exists a hyperbolic set $\Lambda$ for the geodesic flow $\phi^{t}_{_0}\colon  SM \to SM$ with $HD(\Lambda)$  arbitrarily close to $2n-1$.
\end{Le} 
\begin{proof}
Fixed a point $p\in SM$ and consider the family of closed balls, $\Omega_{k}:=B_{k}(p)$, of center $p$ and radius $k\in \mathbb{N}$, .\\
\indent Put $\ds \widetilde{{\Omega}}_{k}=\bigcap _{t\in\mathbb{R}} \phi^{t}_{0}(\Omega_{k})$, then we have the following statement:
$$\ds C \subset \bigcup_{k \in \mathbb{N}} \widetilde{\Omega}_{k},$$
\noindent where $C$ is given in the previous theorem. Indeed, let $x\in C$, then there exists a compact set $K_x$ such that the orbit of $x$, $\ds O(x) \subset K_{x}\subset \Omega_{k_0}$ for some $k_0\in \mathbb{R}^{+}$. This implies that 
$\phi^{t}_{_{0}}(x)\in \Omega_{k_0}$ for all $t\in \mathbb{R}$, therefore $x\in\widetilde{\Omega}_{k_0}$ and the statement is proved.\\


\noindent Now, since $HD(C)=2n-1$, then $\ds\sup_{k\in \mathbb{N}} HD(\widetilde{{\Omega}}_{{k}})=2n-1$, therefore there exists $k_1\in \mathbb{N}$ such that $HD(\widetilde{{\Omega}}_{{k_1}})$ is arbitrarily close to $2n-1$.\\
\indent Notice that $\ds\widetilde{{\Omega}}_{k_1}$ is a compact and $\phi^{t}_{_{0}}$-invariant set. Moreover, since $\phi^{t}_{_{0}}$ is an Anosov flow on $SM$, then $\ds\widetilde{{\Omega}}_{k_1}$ is hyperbolic set for geodesic flow $\phi^{t}_{_{0}}$. Thus, we can take  
\begin{equation}\label{Def of L}
\Lambda:=\ds\widetilde{{\Omega}}_{k_0} \ \text{and} \ HD(\Lambda)\sim 2n-1 \ \ (\text{arbitrarily close to}\ \  2n-1).
\end{equation}

\begin{R}\label{R1-Dodson}
If $N$ is a surface, then the hyperbolic set $\Lambda$, given by the  \emph{Lemma \ref{Dodson}}, has Hausdorff dimension arbitrarily close  to $3$.
Note also that, if $N$ is a $C^r$-Riemannian manifold \emph{(}the Riemannian metric is $C^r$\emph{)} with finite volume, then the \emph{Liouville} measure is preserved by the geodesic flow $\phi^t_{_{0}}$. Therefore, if $\phi_0$ denotes the vector field which derivative  from the geodesic flow, then  $\phi_{_{0}}\in \mathfrak{X}^{r-1}_{w}(SN)$ \emph{(cf. \cite{P})}. 
\end{R}
The proof of the following corollary is based on classical arguments used to construct basis sets. 
\end{proof}
\begin{C}\label{C5} In the case of surface, let $\Lambda$ be the hyperbolic set given by \emph{Lemma \ref{Dodson}}. Then,  there is a basic set $\tilde{\Lambda}$ with $\Lambda \subset\tilde{\Lambda}$. 
\end{C}
\begin{proof}
Note that, by Corollary \ref{C3}, the hyperbolic set $\Lambda$ is one-dimensional, then by similar arguments of Proposition 8 at \cite{Gelf-Burn}, which is based on the argument of Anosov \cite{Anosov2}, we can concluded the proof of corollary. 
\end{proof}
\begin{proof}[\emph{\textbf{Proof of Corollary \ref{C1-Theorem 1}}}]
Simply note that by Remark \ref{R1-Dodson} and Corollary \ref{C5} the basic set $\tilde{\Lambda}$ fits the hypotheses of the Theorem \ref{Theorem 1}, since $HD(\tilde{\Lambda})\geq HD(\Lambda)>2$. In other words, the result of the corollary is an immediate consequence of Theorem \ref{Theorem 1}.
\end{proof}

\subsubsection{Proof of Corollary 2}\label{PC2}
Similar to  Subsection \ref{PC1}, to prove  Corollary \ref{C2-Theorem 1}, we have to find a hyperbolic set with Hausdorff dimension arbitrarily close  to $3$ and then use the Theorem \ref{Theorem 1}. For this purpose, we used the following theorem:\\
\ \\
\textbf{Theorem (Urba\'nski, \cite{Urbanski})}\,\emph{If $M$ is a compact Riemannian manifold and $\varphi \colon M \to M$ is a transitive Anosov diffeomorphism, then the Hausdorff dimension of the set of points with non-dense (full) orbit under $\varphi$ equals dim $M$. The same statement is true for Anosov flows.}\\

As an immediate consequence,
\begin{Le}\label{Le1'}
If $\varphi$ is an Anosov diffeomorphism on a compact surface $N$, then there is a basic  set $\Lambda$ with Hausdorff dimension arbitrarily close to $2$.
\end{Le}
\begin{proof}
Consider $\{x_k\}$ an enumerable and dense set, then for each $m\in \mathbb{N}$, we define the set $A_{m}^{k}:=N\setminus B_{\frac{1}{m}}(x_k)$, where $B_{\frac{1}{m}}(x_k)$ is the open ball of center $x_k$ and radius $\frac{1}{m}$. We consider the compact invariant set  $\tilde{A}_{m}^{k}:= \ds\bigcap_{n\in \mathbb{Z}}\varphi^{n}(A_{m}^{k})$, which is hyperbolic set for $\varphi$, since $\varphi$ is Anosov.
\\
\noindent {Claim:} If $\mathcal{ND}$ is the set of points with non-dense orbit under $\varphi$, then
$$\mathcal{ND}= \ds \bigcup_{m\geq 1} \bigcup_{k}\tilde{A}_{m}^{k}.$$
\begin{proof}[\emph{Proof of Claim}]
We need to prove simply that 
$\mathcal{ND}\subset \ds \bigcup_{m\geq 1} \bigcup_{k}\tilde{A}_{m}^{k}$, indeed:  let $x\in \mathcal{ND}$, then there is an open set $U\subset N$ such that the orbit of $x$,  $O(x)$ does not intersect $U$, \emph{i.e.} $O(x)\cap U=\emptyset$. Thus, there are $m\geq 1$ and $x_k$ such that $B_{\frac{1}{m}}(x_k)\subset U$, therefore $\varphi^{n}(x)\notin B_{\frac{1}{m}}(x_k)$ or $\varphi^{n}(x)\in N\setminus B_{\frac{1}{m}}(x_k)$, for all $n\in \mathbb{N}$, this implies that $x\in \tilde{A}_{m}^{k}.$ 
\end{proof}

\noindent Since Anosov diffeomorphisms on surface are transitive (\cite{K}), then the Urba\'nski's Theorem implies that $HD(\mathcal{ND})=2$, then by the previous Claim, $$2=HD(\ds \bigcup_{m\geq 1} \bigcup_{k}\tilde{A}_{m}^{k})=\ds \sup_{k,m}HD(\tilde{A}_{m}^{k}).$$ 
So, there are $m_{0}$ and $k_{0}$ such that $HD(\tilde{A}_{m_{_{0}}}^{k_0})$ is arbitrarily close to $2$. Note also that $\tilde{A}_{m_{_{0}}}^{k_0}$ is a hyperbolic set which is zero-dimensional, therefore by \cite{Anosov2}, there is a basic set $\Lambda$ such that $\tilde{A}_{m_{_{0}}}^{k_0}\subset \Lambda$ and therefore $HD(\Lambda)$ is arbitrarily close to $2$, as we wanted. 
\end{proof} 


The next step is to prove Corollary \ref{C2-Theorem 1}. For this goal, our task is to use the basic set $\tilde{\Lambda}$ of Lemma \ref{Le1'}  to construct the set of functions $\mathcal{U}_{\varphi}$ of the statement of Corollary \ref{C2-Theorem 1}. The construction of set $\mathcal{U}_{\varphi}$ will be similar to the construction of $\mathcal{U}_{X,\Lambda}$ at section  \ref{Descript of function}, being that we use $N$ instead of $\Xi$. \\

Let $\varphi_0$ be a $C^2$ Anosov diffeomorphism  of a compact surface  $N$ and $\Lambda_0$ the basic set given by the Lemma \ref{Le1'}. Consider $\mathcal{W}_{0}$ a $C^2$ neighborhood of $\varphi_0$ such that, for each $\varphi\in \mathcal{W}_{0}$, the basic set $\Lambda_0$ has a hyperbolic continuation $\Lambda_{\varphi}$. Note that $\Lambda_{\varphi}$ is basic set  and, by the $C^2$-topology, the $HD(\Lambda_{\varphi})$ is also arbitrarily close to $2$ (cf. \cite{PT}).\\
\ \\
Let $\varphi\in \mathcal{W}_{0}$, then by the same notation of Subsection \ref{main-results}, we have 
\begin{Defi}
We say that  $F\in\mathcal{U}_{\varphi}\subset C^2(N_\varphi,\re)$,  whenever
\begin{itemize}
\item[\emph{(i)}] There exists a sub-horseshoe $\Lambda_F$ of $\Lambda_\varphi$ with $HD(\Lambda_F)>1$ and neighborhood $R_F$ of $\Lambda_F$ such that 
$$\textrm{max}F_{\psi_{_{\varphi}}}|_{R_F}\in C^1(R_F,\re),$$
where $\textrm{max} F_{\psi_{_{\varphi}}}(x):=\max_{0\leq t \leq 1}F(\psi_{_{\varphi}}^{t}(x))$.
\item[\emph{(ii)}] $\textrm{max}F_{\psi_{_{\varphi}}}\in H_1(\varphi,\Lambda_{F})$,
where $H_1(\varphi,\Lambda_{F})$ is defined analogously to  $(\ref{E-Descrip})$. 
\end{itemize}
\end{Defi}
Following the same lines of Subsection \ref{Descript of function}, more precisely, the proof of Lemma \ref{set of functions} we have,

\begin{Le}\label{Desc-Susp}
For each $\varphi\in \mathcal{W}_{0}$, the set $\mathcal{U}_{\varphi}$ is dense and $C^2$-open set.
\end{Le}

\begin{proof}[\emph{\textbf{Proof of Corollary  \ref{C2-Theorem 1}}}]
By Theorem \ref{MY1} in Appendix \ref{SIRCS} (see \cite{MY} and \cite{MY1} for more details), we can assume by Lemma \ref{Le1'} that for a small perturbation $\varphi\in \mathcal{W}_{0}$ of $\varphi_0$ in the $C^2$ topology, the pair $(\varphi,\Lambda_\varphi)$ has the property $V$. Let $F\in\mathcal{U}_{\varphi}$, then by the main theorem at \cite{RM}, we have that 
$$\text{int}\, M(\psi,\Lambda_{F}, \textrm{max}F_{\psi_{_{\varphi}}}|_{R_F})\neq \emptyset \ \ \text{and} \ \ \text{int}\, L(\psi,\Lambda_{\psi}, \textrm{max}F_{\psi_{_{\varphi}}}|_{R_F})\neq \emptyset.$$
The proof ends simply by observing that

$$\limsup_{n\to +\infty}\textrm{max}F_{\phi_{_{\varphi}}}(\varphi^n(x))=\limsup_{t\to + \infty}F(\psi^t_{_{\varphi}}(x))$$
and 
$$\sup_{n\to +\infty}\textrm{max}F_{\phi_{_{\varphi}}}(\varphi^n(x))=\sup_{t\to + \infty}F(\psi^t_{_{\varphi}}(x))$$
for all $x\in \Lambda_{\Lambda_F}$.

\end{proof}
To finish this section, we note that, if $\varphi_{0}$ is a $C^2$ conservative Anosov diffeomorphism, then the proof of Corollary \ref{C2-Theorem 1} and Subsection \ref{SBI}, allows us  to conclude that the Remark \ref{R1C2} is valid.

\section{Appendix}\label{Ap}
\subsection{Proof of Separation Lemma \ref{L7}}\label{App - Separation}

In order to prove Lemma \ref{L7} we need to understand what happens when two GCS as in relation (\ref{E6GCS}) intersect.\\
\ \\
Note that if two sections $\Sigma_i$, $\Sigma_j$ has nonempty intersection, then we can consider two disjoint cases: 
\begin{itemize}
\item[1.] The intersection $\Sigma_i\cap\Sigma_j$ is totally transverse to the foliation $\mathcal{F}^{s}$, \emph{i.e.} for any $x\in \text{int}\,\Sigma_i\cap \text{int}\,\Sigma_j$ there is a $C^0$-curve $\xi_x\subset \text{int} \,\Sigma_i\cap \text{int}\,\Sigma_j$ which is  transverse to $\mathcal{F}^{s}$, then the Proposition \ref{P3} implies that $\textrm{int}\,\Sigma_i\cap \textrm{int}\,\Sigma_j$ is an open set of $\Sigma_i$ and $\Sigma_j$. 
\item[2.] The intersection $\Sigma_i\cap\Sigma_j$ does not is totally transverse to the foliation $\mathcal{F}^{s}$, \emph{i.e.}, there may be points in $\Sigma_i\cap\Sigma_j$ in the following two situations: 
\begin{itemize}
\item[(i)] For every point $x \in \Sigma_i\cap\Sigma_j$ there is not a curve $\xi_x\subset \textrm{int} \,\Sigma_i\cap \textrm{int}\,\Sigma_j$ transverse to $\mathcal{F}^{s}$, this implies $\Sigma_i\cap\Sigma_j$ does not contains open sets of $\Sigma_i$ and $\Sigma_j$. 
\item[(ii)] The intersection $\Sigma_i\cap\Sigma_j$  contains an open set of $\Sigma_i$ and $\Sigma_j$ and also contains  points as in (i).
\end{itemize}
\end{itemize}

The next task is to understand the  cases (i) and (ii). First, we let us make the separation in the Lemma \ref{L7} when all intersections of the sections $\Sigma_{i}$ satisfies condition $1$. After that, we will make the separation in Lemma \ref{L7} when appears  intersections in conditions $1$ or $2$.

\begin{Le}\label{L6GCS}
Let $B_i=\{j: \Sigma_i\cap \Sigma_j\neq \emptyset \ \ \text{and} \ \ \Sigma_{i}, \, \Sigma_{j} \ \ \text{satisfies the condition 1} \}$.
 Then there is $\delta^{\prime}>0$ such that for every $j\in B_i$, $\phi^{\delta}(\Sigma_{i})\cap \Sigma_{j}=\emptyset$ for all $0<\delta\leq\delta^{\prime}$. 
\end{Le}
\begin{proof}
Suppose otherwise, then for all $n$ sufficiently large, there is $z_{i}^{n}\in \Sigma_{i}$ such that $\ds\phi^{\frac{1}{n}}(z^{n}_{i})\in \bigcup_{j\in B_i}\Sigma_{j}$. Passing to a subsequence if necessary, we assume that $\phi^{\frac{1}{n}}(z^{n}_{i})\in \Sigma_{j_0}$ for some $j_0\in B_i$. Since $\Sigma_{i}$ is a compact set, we can assume that $z^{n}_{i}$ converge to $z_{i}$ as $n\to\infty$, thus $\phi^{\frac{1}{n}}(z^{n}_{i})$ converge to $z_{i}$ as $n\to \infty$. This implies that $z_{i}\in \Sigma_{i}\cap \Sigma_{j_0}$.\\
\noindent By Remark \ref{R2'}, we can assume that $z_{i}\in \text{int}\,\Sigma_{j_0}$, then by definition of $C^0$-transverse (Definition \ref{Def1}), there are $r>0$ small enough and $\eta>0$ such that $B_{r}(z_{i})$ (the open ball of radius $r$ and center $z_{i}$), satisfies 
\begin{center}
$\phi^{t}(B_{r}(z_i)\cap \Sigma_{j_0})\cap \Sigma_{j_0}=\emptyset$
\end{center}
for all $0<t\leq\eta$.\\
Moreover, since $\Sigma_{i}$ and $\Sigma_{j_0}$ satisfies the condition 1, then  we have $(B_{r}(z_i)\cap \Sigma_{i})\setminus\{z_i\}\subset \Sigma_{j_0}$. Taking  $n$ large enough such that $z^{n}_{i}\in B_{r}(z_i)\cap \Sigma_{i}$ and $\frac{1}{n}<\eta$. So $\phi^{\frac{1}{n}}(z^{n}_{i})\notin \Sigma_{j_0}$ which is a contradiction, thus we concluded the lemma.\\
\end {proof}
\begin{R}\label{R5'}
It is worst to note that, if $\delta_{ij}:=d(\Sigma_i,\Sigma_j)>0$, then 
$$\phi^{t}(\Sigma_i)\cap\Sigma_j=\emptyset \ \ \text{for all} \ \ 0\leq t< \delta_{ij}.$$
\end{R}
\noindent The following lemma proves that the GCS as in (\ref{E6GCS}) can be taken disjoint if all possible intersections of $\Sigma_{i}$ and $\Sigma_{j}$ satisfy the condition $1$.\\

\begin{Le}\label{L7GCS}
Assuming $(\ref{E6GCS})$ and suppose that all possible intersections of  sections $\{\Sigma_{i}:i=1,\dots,l\}$ satisfies the condition 1. Then, there are GCS\, $\widetilde{\Sigma}_{i}$ such that $\ds \Lambda\subset \bigcup^{l}_{i=1}\phi^{(-\gamma,\gamma)}(\widetilde{\Sigma}_{i})$ with the property $\widetilde{\Sigma}_{i}\cap\widetilde{\Sigma}_{j}=\emptyset$ for all $i,j\in\{1,...,l\}$.
\end{Le}
\begin{proof}
Consider the set $B_1=\{j:\Sigma_{1}\cap \Sigma_{j}\neq\emptyset\}$ and $\ds\delta_1=\inf_{j\notin B_1}d(\Sigma_1,\Sigma_j)$, then by Lemma \ref{L6GCS}, there exist $t_1< \min \{\delta_1, \gamma\}$ such that 
$$\phi^{t_1}(\Sigma_1)\cap \Sigma_j=\emptyset \ \ \text{for all} \ \ j\geq 1.$$
Put $\widetilde{\Sigma}_{1}:=\phi^{t_1}(\Sigma_1)$ and $\beta_2=d(\widetilde{\Sigma}_{1},\Sigma_2)$. Analogously, we consider the set $B_2=\{j\geq 2:\Sigma_{2}\cap \Sigma_{j}\neq\emptyset\}$
and $\ds\delta_2=\min \{\inf_{j\notin B_2}d(\Sigma_2,\Sigma_j), \beta_2\}$, then by Lemma \ref{L6GCS}, there exist $t_2< \min \{\delta_2, \gamma\}$ such that 
$$\phi^{t_2}(\Sigma_2)\cap \Sigma_j=\emptyset \ \ \text{for all} \ \ j\geq 2 \ \ \text{and} \ \ \phi^{t_2}(\Sigma_2)\cap \widetilde{\Sigma}_{1}=\emptyset.$$
We can continue with this process and obtain by induction a finite sequences of positive number $\delta_i, \beta_i$ and $t_i$ define by $\beta_i=\ds\min_{1\leq j<i}d(\phi^{t_j}(\Sigma_j),\Sigma_i)$, $\ds\delta_i=\min \{\inf_{j\notin B_i}d(\Sigma_i,\Sigma_j), \beta_i\}$, where $B_i=\{j\geq i:\Sigma_{i}\cap \Sigma_{j}\neq\emptyset\}$ and $t_i<\min \{\delta_i, \gamma\}$ with the properties
$$\phi^{t_i}(\Sigma_i)\cap \Sigma_j=\emptyset \ \ \text{for all} \ \ j\geq i \ \ \text{and} \ \ \phi^{t_i}(\Sigma_i)\cap \phi^{t_j}({\Sigma}_{j})=\emptyset \ \ \text{for all} \ \ j\leq i.$$
Put $\widetilde{\Sigma}_{i}:=\phi^{t_i}(\Sigma_i)$, then it is easy to see that the set of sections $\{\widetilde{\Sigma}_{1}, \widetilde{\Sigma}_{2},\dots, \widetilde{\Sigma}_{l}\}$ satisfies the conditions of lemma.
\end{proof}

The following lemma show that if two GCS satisfy  the condition 2\,(ii), then a small translate in the time on one of two sections makes the resulting sections satisfy the condition 2\,(i).

\begin{Le}\label{L7'0}
Let $\Sigma_i$, $\Sigma_j$ be as in \emph{(\ref{E6GCS})} satisfying  the condition ${2\emph{(ii)}}$, then there is $t'$ small such $\phi^{t'}(\Sigma_i)$ and $\Sigma_j$ satisfy the condition $2\,\emph{(i)}$, \emph{i.e.},  $\phi^{t'}(\Sigma_i) \cap \Sigma_j$ does not contain an open set of $\phi^{t'}(\Sigma_i)$ nor $\Sigma_j$.
\end{Le}
\begin{proof}
By contradiction, assume that for  all $t$ small enough,  we have that $\textrm{int} \, \phi^{t}(\Sigma_i)\cap \textrm{int} \, \Sigma_j$ contains an open set of $\phi^{t}(\Sigma_i)$ and $\Sigma_j$, then there is a non-degenerate interval $I^{j}_{t}\subset W_{\epsilon}^{u}(x_j)\subset \Sigma_j$ and a non-degenerate interval $I^i_t\subset {W^{u}_{\epsilon}}(x_i)\subset \Sigma_i$
such that the set 
\begin{equation}\label{EQ5'}
\Delta_t:=\bigcup_{z\in \phi^{t}(I^{j}_{t}))}W^s_{\epsilon}(z)\cap \bigcup_{w\in I^{i}_{t}}W^s_{\epsilon}(w)
\end{equation}
contains an open set of $\phi^{t}(\Sigma_i)$ and $\Sigma_j$.\ \\

\noindent \textbf{Claim:} The family of intervals $I^{j}_{t}$ is pairwise disjoint. 
\begin{proof}[\textbf{\emph{Proof of claim}}]
Otherwise, assume that there is $x\in I^{j}_{t}\cap I^j_{t'}\subset W^{u}_{\epsilon}(x_j)$ with $t\neq t'$, then by (\ref{EQ5'}) there are $y\in I^{i}_{t}$ and $z\in I^{i}_{t'}$ such that $\phi^{t}(x)\in W^{s}_{\epsilon}(y)$ and $\phi^{t'}(x)\in W^{s}_{\epsilon}(z)$. Since $t, t'$ are small, then we have that $\phi^{-t'}(z)\in W^{s}_{loc}(\phi^{-t}(y))$. Also, since $y, \, z \in W^{u}_{\epsilon}(x_i)$, then $z\in W^{s}_{2\epsilon}(y)$, which implies that $\phi^{-t'}(z)\in W^{u}_{loc}(\phi^{-t'}(y))$. Since $t'-t$ is also small, then we have 
$$\phi^{-t'}(z)\in \phi^{t'-t}(W^{u}_{loc}(\phi^{-t'}(y)))\cap W_{loc}^{s}(\phi^{-t'}(y)),$$
which implies that $t'=t$ and $z=y$, since the stable and unstable manifold theorem, which is a contradiction. 
\end{proof}


\noindent Note that the  above  claim provides a contradiction because does not exists uncountable many disjoint non-empty open intervals $I_t^j$. Since each of them would contain a rational number, proving an uncountable family of distinct rational  numbers. So the proof is complete.

\end{proof}
\begin{R}The last lemma implies that: always we can assume that the sections  $\Sigma_i$ and $\Sigma_j$ satisfies the condition  1) or $2\,(i)$. 
\end{R}
\noindent The next step is to understand what happens to the cross-sections that intersect as in case { 2(i)}. \\

Assume that $\Sigma_i$  and $\Sigma_j$ satisfy the condition 2(i), then  $\Sigma_i \cap \Sigma_j$ is a family of curves, $\Gamma_{ij}$. By  construction of GCS of Lemma \ref{Le6}, each curve  $c\in \Gamma_{ij}$ is a leaf of the foliation $\{\mathcal{F}^s(x)\cap \Sigma_i : x\in W^u_{\epsilon}(x_i)\}$, by abuse of notation we write $\mathcal{F}^{s}\cap \Sigma_i$. Remember that $\Sigma_i=\Sigma_{x_i}$, thus we  consider the projection $\pi^{s}_i\colon \Sigma_i\to W^u_{\epsilon}(x_i)$  along $\mathcal{F}^{s}$.
\begin{Pro}\label{P2}
In the above conditions $\pi^{s}_i(\Sigma_i\cap \Sigma_j)= \{W^{s}_{\epsilon}(x)\cap  W^u_{\epsilon}(x_i)\colon x\in \Sigma_i \cap \Sigma_j\}$ is a compact set.
\end{Pro}
\begin{proof}
 Indeed, we need only to show that $\pi^{s}_i(\Sigma_i\cap \Sigma_j)$ is a closed set. Let $x_n\in \pi^{s}_i(\Sigma_i\cap \Sigma_j)$, such that $x_n \to x$, then there is $y_n\in \Sigma_i\cap \Sigma_j$ such that $W^{s}_{\epsilon}(y_n)\cap  W^u_{\epsilon}(x_i)=\{x_n\}$, since $\Sigma_i\cap \Sigma_j$ is compact, then we can assume that $y_{n_k} \to y\in \Sigma_i\cap \Sigma_j$. Moreover, by continuity of foliation $\mathcal{F}^s$, we have that $W^{s}_{\epsilon}(y_{n_k})\cap  W^u_{\epsilon}(x_i)\to \pi^{s}_{i}(y)$ and $W^{s}_{\epsilon}(y_{n_k})\cap  W^u_{\epsilon}(x_i)=\{x_{n_k}\}\to x$, so $x=\pi^{s}_{i}(y)\in \pi^{s}_i(\Sigma_i\cap \Sigma_j)$.
 \end{proof}
 
 \begin{R}
It is worth noting that the proof of the previous proposition,  actually shows that $\pi^{s}_i$ is a continuous map.
 \end{R}

Let $x \in\pi^{s}_i(\Sigma_i\cap \Sigma_j)$, then, by  transversality of the flow with  both sections, there is $\delta>0$ such that   
$$\phi^{\delta}(W^s_{\epsilon}(x)\cap \Sigma_i)\cap \Sigma_j=\emptyset,$$
and by continuity we have that there is $U_x$ neighborhood of $W^s_{\epsilon}(x)\cap \Sigma_i$ on $\Sigma_i$ such that 
\begin{equation}\label{eq-sep}
\phi^{\delta}(U_x\cap \Sigma_i)\cap \Sigma_j=\emptyset.
\end{equation}
\noindent The neighborhood $\displaystyle U_x := \bigcup_{z\in I_x} W^{s}_{\epsilon}(z)\cap \Sigma_i$, where  $I_{x}\subset W^{u}_{\epsilon}(x_i)$ is an interval centered in $x$.\\
\indent Suppose that $\mathcal{F}^{s}(x)\cap \Sigma_{i}\cap \Lambda=\emptyset$ for some $x\in\pi^{s}_i(\Sigma_i\cap \Sigma_j)$, then since $\Lambda$ is a compact set there is an open set $V_{x}$ containing $\mathcal{F}^{s}(x)\cap \Sigma_i$ with $V_{x}\cap\Lambda=\emptyset$. Therefore, $\Sigma_{i}$ can be subdivided into two GCS, $\Sigma_{i}^{1}$ and $\Sigma_{i}^{2}$, such that $\Sigma_{i}^{r}$ and $\Sigma_j$ intersecting  as the case 2(i) for $r=1,2$. In other words, if $\mathcal{F}^{s}(x)\cap \Sigma_{i}\cap \Lambda=\emptyset$ for some $x\in\pi^{s}_i(\Sigma_i\cap \Sigma_j)$, then we return to the case 1) or 2(i) with one more section. 
\begin{R}\label{R6}
The above observation implies that, without loss of generality, we can assume that for any $x\in\pi^{s}_i(\Sigma_i\cap \Sigma_j)$ there is $p_{x}\in \mathcal{F}^{s}(x)\cap \Sigma_i\cap \Lambda$.
\end{R}

\begin{Le}\label{L8GCS}
If $\Sigma_i$ and $\Sigma_j$ are two GCS as in condition \emph{2(i)}. Given $\delta>0$, $0<\delta<\frac{\gamma}{2}$ $($with $\gamma$ as in \emph{(\ref{E6GCS})}$)$, then for $x\in \pi^{s}_i(\Sigma_i\cap \Sigma_j)$, 
there is a $GCS$, $\widetilde{\Sigma}_{x}\subset U_x\cap{\Sigma_i}$ containing $\mathcal{F}^{s}(x)\cap \Sigma_i$, such that $\Sigma_{i}$ is subdivided into three disjoint GCS, including $\widetilde{\Sigma}_{x}$. Denoted by  $\Sigma_{i}^{\#}$ the set of complementary sections of $\widetilde{\Sigma}_{x}$ in the above subdivision of $\Sigma_{i}$, then  
\begin{enumerate}
	\item[$1)$] $\phi^{\delta}(\widetilde{\Sigma}_{x})\cap \Sigma_j=\emptyset$.
	\item[$2)$] $ \Lambda\cap \phi^{(-\frac{\gamma}{2},\frac{\gamma}{2})}(\emph{int}(\Sigma_i))\subset \Lambda \cap  \left(\phi^{(-\gamma,\gamma)}\left(\phi^{\delta}(\emph{int}(\widetilde{\Sigma}_{x}))\right)\cup \displaystyle \bigcup_{\Sigma\in \Sigma_{i}^{\#}}\phi^{(-\frac{\gamma}{2},\frac{\gamma}{2})}(\emph{int}(\Sigma))\right).$
\end{enumerate}
\end{Le}
\begin{proof}
\noindent By Remark \ref{R6}, we can assume that for any $x\in\pi^{s}_i(\Sigma_i\cap \Sigma_j)$ there is $p_{x}\in \mathcal{F}^{s}(x)\cap \Sigma_i\cap \Lambda$. Consider $\mathcal{F}^{u}_{loc}(p_{x})$, then by Remark \ref{R3} we can find open sets $V_{p_x}^{+}$ and $V_{p_x}^{-}$ in each side of $\mathcal{F}^{u}_{loc}(p_{x})\setminus\{p_x\}$ sufficiently close to $\mathcal{F}^{s}_{loc}(p_{x})$ with diameter sufficiently large and $V_{p_x}^{\pm}\cap \Lambda=\emptyset$. Denote by $\widetilde{V}_{p_x}^{\pm}$ the projection  along to the flow of $V_{p_x}^{\pm}$ over $\Sigma_{i}$, respectively.
Therefore, by Remark \ref{R3} we can take $\widetilde{V}_{p_x}^{\pm}$ such that $\widetilde{V}_{p_x}^{\pm}\cap \Sigma_{i}\subset U_x$ and $\widetilde{V}^{\pm}_{p_x}$ crosses $\Sigma_{i}$. Using $\widetilde{V}^{\pm}_{p_x}$ we can construct the GCS $\widetilde{\Sigma}_{x}$ such that $\widetilde{\Sigma}_{x}\subset U_x$ and by (\ref{eq-sep}), $\widetilde{\Sigma}_{x}$ satisfies the item  1) of lemma.\\
\noindent To prove item 2) note simply that $\delta<\frac{\gamma}{2}$ and $\ \phi^{(-\frac{\gamma}{2},\frac{\gamma}{2})}(\text{int}(\Sigma_i))\cap \Lambda=\phi^{(-\frac{\gamma}{2},\frac{\gamma}{2})}( \text{int}(\Sigma_{i})\cap\Lambda)$, which is a consequence of $\Lambda$ be invariant by the flow.
\end{proof}
\indent As the GCS $\widetilde{\Sigma}_{x}$ obtained in the last lemma is contained in $U_x\cap{\Sigma_i}$, then there is an interval centered in $x$, $\widetilde{I}_x\subset I_x\subset W^{u}_{\epsilon}(x_i)$ such that 
$$\widetilde{\Sigma}_x = \ds\bigcup_{z\in\widetilde{I}_x}W^{s}_{\epsilon}(z)\cap \Sigma_i.$$
Moreover, since  $\pi^{s}_i(\Sigma_i\cap \Sigma_j)\subset \bigcup \widetilde{I}_{x}$, then  the compactness $\pi^{s}_i(\Sigma_i\cap \Sigma_j)$ from Proposition \ref{P2},  there is a finite set of points $\{x^1,\dots,x^m\}\subset \pi^{s}_i(\Sigma_i\cap \Sigma_j)$  such that
$$\pi^{s}_i(\Sigma_i\cap \Sigma_j)\subset \bigcup_{r=1}^{m} {\widetilde{I}}_{x^r}.$$
Thus by first part of Lemma \ref{L8GCS}, given $\delta>0$ small enough,  holds that  
\begin{equation}\label{Eq9}
\phi^{\delta}(\widetilde{\Sigma}_{x^{r}})\cap \Sigma_j=\emptyset,  \ \ \ r=1,\dots,m \ \ \text{and} \ \ \phi^{\delta}(\widetilde{\Sigma}_{x^{r}})\cap\phi^{\delta}(\widetilde{\Sigma}_{x^{r'}})=\emptyset, \ \ r\neq r'.
\end{equation}
In the above conditions, we prove the following 
\begin{Le}\label{L9GCS}
If $\Sigma_i$ and $\Sigma_j$ are two GCS as in the condition \emph{2(i)}. Given $0<\delta<\frac{\gamma}{2}$ 
$($with $\gamma$ as in $($\emph{\ref{E6GCS}}$))$ there are GCS, $\widetilde{\Sigma}_{x^r}\subset U_{x^r}$ containing $\mathcal{F}^{s}(x^r)\cap \Sigma_i$, $r=1,\dots, m$ and such that $\Sigma_{i}$ is subdivided into $2m+1$ disjoint GCS, including 
$\widetilde{\Sigma}_{x^r}$, $r\in\{1,\dots,m\}$. Denote by  $\Sigma_{i}^{\#}$ the complement of the set $\{\widetilde{\Sigma}_{x^r}\}_{r=1}^{m}$ in the above subdivision  of $\Sigma_{i}$, then  
\begin{enumerate}
	\item[$1)$] $\phi^{\delta}(\widetilde{\Sigma}_{x^r})\cap \Sigma_j=\emptyset$, $r\in\{1,\dots,m\}$ and $\Sigma_j\cap \Sigma=\emptyset$ for any $\Sigma\in \Sigma_{x}^{\#}$.
	\item[$2)$] $\phi^{\delta}(\widetilde{\Sigma}_{x^r})\cap\phi^{\delta}(\widetilde{\Sigma}_{x^{r'}})=\emptyset$, $r\neq r'$ and
	$\phi^{\delta}(\widetilde{\Sigma}_{x^r})\cap \Sigma_{i}=\emptyset$, $r\in\{1,\dots,m\}$.
	\item[$3)$] $\Lambda\cap \phi^{(-\frac{\gamma}{2},\frac{\gamma}{2})}(\emph{int}(\Sigma_i))\subset \Lambda \cap \ds \left(\bigcup^{m}_{r=1}\phi^{(-\gamma,\gamma)} \left(\phi^{\delta}(\emph{int}(\widetilde{\Sigma}_{x^r}))\right)\cup \bigcup_{\Sigma\in \Sigma_{i}^{\#}}\phi^{(-\frac{\gamma}{2},\frac{\gamma}{2})}(\emph{int}(\Sigma))\right).$
\end{enumerate}
\end{Le}
\begin{proof}
Given $0<\delta<\frac{\gamma}{2}$ small enough. The conditions 1) and 2) are  an immediate consequence of (\ref{Eq9}). To prove item 3) note simply that $\delta<\frac{\gamma}{2}$ and $\Lambda\cap \phi^{(-\frac{\gamma}{2},\frac{\gamma}{2})}(\text{int}(\Sigma_i))=\phi^{(-\frac{\gamma}{2},\frac{\gamma}{2})}(\Lambda\cap \text{int}(\Sigma_{i}))$, which is a consequence of $\Lambda$ be invariant by the flow.
\end{proof}

\begin{R}\label{R6GCS}
Let $\Sigma'$ be such that GCS such that $\Sigma'\cap \Sigma_i=\emptyset$. Then we can take $\delta<d(\Sigma',\Sigma_i)$, such that  $\phi^{\delta}(\widetilde{\Sigma}_{x^r})\cap\Sigma'=\emptyset$, \, $r\in\{1,\dots,m\}$, where $\widetilde{\Sigma}_{x^r}$ as in the $Lemma \, \ref{L9GCS}$.
\end{R}
To give a complete proof of Lemma \ref{L7}, we must now treat the more general case of Lemma \ref{L9GCS}, where three or more sections intersect as in case 2i)\\
\ \\
\noindent To reinforce the idea,  we recall the equation  (\ref{E6GCS})
\begin{eqnarray*}
\ds \Lambda\subset  \bigcup^{l}_{i=1}\phi^{(-\gamma,\gamma)}(\textrm{int}\,{\Sigma_{i}})=\bigcup^{l}_{i=1}U_{\Sigma_{i}}.
\end{eqnarray*}


\noindent Now we will prove that the GCS in (\ref{E6GCS}) can be taken disjoint, even if some of the cross-sections  are in  condition 2(i).
\begin{Le}\label{L10GCS}
Let $\Sigma_{i}$ be a GCS as in $(\ref{E6GCS})$. Let $B_{i}=\{j: \Sigma_{i} \ \ \text{intersects} \ \  \Sigma_{j} \ \ \text{as the case} \, \, \, \emph{2(i)}\}$. Then, $\Sigma_{i}$ can be subdivided in a finite number of GCS $\{\Sigma_{i}^{s}:s=1,\dots,n\}$ such that for each $s$, there is $0<\delta_{s}<\frac{\gamma}{2}$ such that 
\begin{enumerate}
	\item[$1)$] $\phi^{\delta_{s}}(\Sigma_{i}^{s})\cap\Sigma_{j}=\emptyset$, $j\in B_{i}$ and $\ds \phi^{\delta_{s}}(\Sigma_{i}^{s})\cap \phi^{\delta_{s'}}(\Sigma_{i}^{s'})=\emptyset$,  $s\neq s'$.
	\item[$2)$] $\ds \Lambda \cap \bigcup_{j\in B_{i}\cup \{i\}}\phi^{(-\frac{\gamma}{2},\frac{\gamma}{2})}(\emph{int}{\Sigma_{j}})\subset \Lambda \cap \left(\bigcup_{j\in B_{i}}\phi^{(-\frac{\gamma}{2},\frac{\gamma}{2})}(\emph{int}\Sigma_{j})\cup\bigcup_{s=1}^{n} \phi^{(-\gamma,\gamma)}\left(\emph{int}(\phi^{\delta_s}(\Sigma_{i}^{s}))\right)\right)$.
\end{enumerate}
\end{Le}
\begin{proof}
The proof is by induction on $\#B_{i}$. The case $\#B_{i}=1$ is true by the Lemma \ref{L8GCS}. Suppose the statement is true for $\#B_{i}<q$ and we prove for $\#B_{i}=q$. In fact:\\
Let $k\in B_{i}$, then by Lemma \ref{L8GCS}, given $0<\delta<\frac{\gamma}{2}$, there are a finite number of GCS $\left\{\widetilde{\Sigma}^{r}_{k}\subset \Sigma_{k}: r\in\{1,\dots,r_{k}\}\right\}$ such that 

 \begin{eqnarray}\label{E8GCS}
	\ds \phi^{\delta}(\widetilde{\Sigma}^{r}_{k})\cap \Sigma_{k}=\emptyset, \ \text{also} \ \ \ds \phi^{\delta}(\widetilde{\Sigma}^{r}_{k})\cap \Sigma_{i}=\emptyset \ \text{for any} \ r, \ \text{and} \  \Sigma_{i}\cap\Sigma=\emptyset, \  \ \Sigma\in\Sigma_{k}^{\#}.
	\end{eqnarray}
 \begin{eqnarray}\label{E9GCS} \Lambda\cap \phi^{(-\frac{\gamma}{2},\frac{\gamma}{2})}(\text{int}\,\Sigma_k)\subset \Lambda \cap \left(\bigcup^{r_k}_{r=1}\phi^{(-\gamma,\gamma)} \left(\phi^{\delta}(\text{int}\,\widetilde{\Sigma}_{k}^r)\right)\cup \bigcup_{\Sigma\in \Sigma_{k}^{\#}}\phi^{(-\gamma,\gamma)}(\text{int}\,\Sigma)\right),
	\end{eqnarray}
where  $\Sigma_{k}^{\#}$ is as in the Lemma \ref{L8GCS}.\\
\noindent Consider now the collection  of  GCS 
$$\left\lbrace \Sigma_{i},\Sigma_{j}, \phi^{\delta}(\widetilde{\Sigma}_{k}^r), \Sigma_{k}^{\#}:j\in B_{i}\setminus \{k\} \ \text{and} \ \ r\in\{1,\dots,r_{k}\}\right\rbrace.$$

\noindent For this new collection of GCS, we have $\#B_{i}<q$. Therefore, by the induction hypothesis, the lemma is true for $\left\{\Sigma_j:j\in B_{i}\cup \{i\}\setminus\{k\}\right\}$ and by (\ref{E8GCS}) and (\ref{E9GCS}) we have the lemma.
\end{proof}
\begin{R}\label{R7GCS}
Let $\Sigma'$ be a GCS as in $(\ref{E6GCS})$ such that $\Sigma'\cap \Sigma_{i}=\emptyset$. Then by Remark \ref{R6GCS}, we can take $\delta_{s}$ less than $d(\Sigma_{i},\Sigma')$. So $\phi^{\delta_{s}}(\Sigma_{i}^{s})\cap \Sigma'=\emptyset$ for any $s\in\{1,\dots,m\}$, $\Sigma_{i}^{s}$ as in the Lemma $\ref{L10GCS}$.
\end{R}
We finish this section making the proof of the Lemma \ref{L7}.
\begin{proof}[\textbf{\emph{Proof of Lemma \ref{L7}}}]
 If all the possible intersections satisfy condition 1, the result follows from Lemma \ref{L7GCS}. Then, we can suppose that there is $i$, such that the set  $B_{i}=\{j: \Sigma_{i} \ \ \text{intersects} \ \  \Sigma_{j} \ \ \text{as the case}\, \,  \text{2(i)}\}$ is non-empty.
Without loss of generality, assume that $B_{1}\neq\emptyset$. Let us will conclude the proof by induction on  $l$ at (\ref{E6GCS}).\\
\indent Note that the Lemma \ref{L8GCS} implies the result in the case $l=2$.  Therefore, suppose it is true for $k<l$ and we will prove for $k=l$. Indeed, fix $\Sigma_{1}$ and consider the set  $$T_{1}=\{j: \Sigma_{j} \ \ \text{intersects} \ \  \Sigma_{1} \ \ \text{as the case} \, \, \, \text{1}\}.$$ Then by Lemma \ref{L6GCS}, there is $0<\delta<\frac{\gamma}{2}$ small enough, such that $\phi^{\delta}(\Sigma_{1})\cap \Sigma_{j}=\emptyset$ for any $j\in T_{1}$.\\
\indent  Abusing the notation, let's still call  $B_1=\{j:\phi^{\delta}(\Sigma_{1}) \ \ \text{intersects} \ \  \Sigma_{j} \ \ \text{as the case} \, \, \text{2(i)}\}$. Then by Lemma \ref{L9GCS}, $\phi^{\delta}(\Sigma_{1})$ can be subdivided in a finite number of GCS $\{\Sigma_{1}^{s}:s=1,\dots,m_0\}$ and for each $s$ there is $0<\delta_{s}<\frac{\gamma}{2}$ such that holds $1)$ and $2)$ of Lemma \ref{L9GCS}. Also by Remark \ref{R7GCS}, we can assume that $\phi^{\delta_s}(\Sigma_{1}^{s})\cap \Sigma_{j}=\emptyset$ for any $s\in\{1,\dots,m_0\}$ and for any $j\in T_{1}\setminus \{1\}$.\\
\indent Since the cardinal $\#\left(T_{1}\setminus\{1\}\cup B_{1}\right)<l$, then the set $\ds\left\{\Sigma_{j}:j\in T_{1}\setminus\{1\}\right\}\cup \left\{\Sigma_{k}:k\in B_{1}\right\}$ satisfies the induction hypothesis, therefore  there are $n(l)-1$ GCS, $\widetilde{\Sigma}_{i}$ with $\widetilde{\Sigma}_{i}\cap\widetilde{\Sigma}_{j}=\emptyset$ for $i\neq j$, such that 
\begin{eqnarray}\label{E10GCS}
\Lambda\cap \bigcup_{i\in T_1\cup B_1\setminus\{1\}}\left(\phi^{-(\gamma,\gamma)}(\textrm{int}\, \Sigma_i)\right)\subset \Lambda\cap \bigcup^{n(l)}_{i=2}\phi^{(-2\gamma,2\gamma)}(\textrm{int}\,\widetilde{\Sigma}_i).
\end{eqnarray}

\noindent Since $\phi^{\delta_s}(\Sigma_{1}^{s})\cap \Sigma_{j}=\emptyset$ for any $j\in T_1\cup B_1\setminus \{1\}$ and any $s\in\{1,\dots,m\}$, then the $\widetilde{\Sigma}_j$ may be taken such that $\phi^{\delta_s}(\Sigma_{1}^{s})\cap\widetilde{\Sigma}_i=\emptyset$ for any $s\in\{1,\dots,m\}$ and any $i\in \{2,\dots,n(l)\}$.\\
So, by the condition $2)$ of Lemma \ref{L9GCS} and (\ref{E10GCS}) we have that
\begin{eqnarray*}
 \ds\Lambda&=&\Lambda\cap \bigcup^{l}_{j=1}\phi^{(-\gamma,\gamma)}(\textrm{int}\,\Sigma_j)\subset \Lambda \cap\left( \bigcup^{l}_{j=2}\phi^{(-\gamma,\gamma}(\textrm{int}\,\Sigma_{j})\cup \phi^{(-\gamma,\gamma)}\left(\textrm{int}\,\phi^{\delta}(\Sigma_{1})\right)\right)\\
&=& \ds\Lambda \cap \left(\bigcup_{j\in B_{1}}\phi^{(-\gamma,\gamma)}(\textrm{int}\,\Sigma_{j})\cup\bigcup_{j\in T_{1}\setminus\{1\}}\phi^{(-\gamma,\gamma)}(\textrm{int}\,\Sigma_{j})\cup\phi^{(-\gamma,\gamma)}\left(\textrm{int}\,\phi^{\delta}(\Sigma_{1})\right)\right)\\
&\subset & \Lambda\cap \left(\bigcup^{n(l)}_{i=2}\phi^{(-2\gamma,2\gamma)}(\textrm{int}\,\widetilde{\Sigma}_i)\cup \bigcup_{s=1}^{m_0} \phi^{(-2\gamma,2\gamma)}\left(\textrm{int}\,\phi^{\delta_s}(\Sigma_{i}^{s})\right)\right).
\end{eqnarray*}
Therefore, the last inclusion  concludes our proof, since  $m=n(l)-1+m_0$. 

\end{proof}

\subsection{Proof of Hyperbolicity of Poincar\'e Map}\label{PHPM}
Our main goal of this section is to prove Lemma \ref{LHPM}. We recall some information. 
Let $\Xi=\bigcup_{i=1}^{m} \Sigma_{i}$ be a finite union of cross-sections to the flow $\phi^{t}$ given by  Remark \ref{R11}, which are pairwise disjoint. Sometimes, abusing of notation,  we consider $\Xi=\{\Sigma_1,\cdots, \Sigma_l\}$. Let ${\cR}\colon \Xi \to \Xi$ be a Poincar\'e map, that is, the map of first return to $\Xi$, ${\cR}(y)=\phi^{t_{+}(y)}(y)$, where $t_{+}(y)$ corresponds to the first time that the positive orbits of $y\in \Xi$ encounter $\Xi$. \\
The splitting $E^{s}\oplus \phi\oplus E^{u}$ over a neighborhood $U_{0}$ of $\Lambda$ defines a continuous splitting $E^{s}_{\Sigma}\oplus E^{u}_{\Sigma}$ of the tangent bundle $T\Sigma$ with $\Sigma\in \Xi$ given by 
\begin{eqnarray}\label{eq7}
E^{s}_{\Sigma}(y)=E^{cs}_{y}\cap T_{y}\Sigma \ \text{and} \ E^{u}_{\Sigma}(y)=E^{cu}_{y}\cap T_{y}\Sigma,
\end{eqnarray}
where $E_{y}^{cs}=E^{s}_y\oplus \left\langle \phi(y)\right\rangle$ and $E_{y}^{cu}=E^{u}_y\oplus \left\langle {\phi}(y)\right\rangle$.\\

We will show that for a sufficiently large iterated of ${\cR}$, ${\cR}^{n}$, the splitting (\ref{eq7}) defines a hyperbolic splitting for transformation ${\cR}^{n}$ on the cross-sections, at least restricted to $\Lambda\cap \Xi$ (cf. \cite[chap. 6]{VP}). To achieve this goal, we will take into consideration the following:
\begin{R}\label{R2GCS}
\item[$(1)$] In what follows, we use $K\geq1$ as a generic notation for large constants
depending only on a lower bound for the angles between the cross-sections and the flow
direction. Also depending on upper and lower bounds for the norm of the vector field on the cross-sections.
\item[$(2)$] Let us consider unit vectors, $e^{s}_{x}\in E^{s}_{x}$ and $\hat{e}^{s}_{x}\in E^{s}_{\Sigma}(x)$, and write
\begin{eqnarray}\label{E2GCS}
e^{s}_{x}=a_{x}\hat{e}^{s}_{x}+b_{x}\frac{\phi(x)}{\left\|\phi(x)\right\|}.
\end{eqnarray}

\noindent Since the angle between $E^{s}_{x}$ and $\phi(x)$,  $\angle (E^{s}_{x},\phi(x))$, is greater than or equal to the angle between $E^{s}_{x}$ and $E^{cu}_{x}$,  $\angle(E^{s}_{x},E^{cu}_{x})$, due to the fact $\phi(x)\in E^{cu}_{x}$. The latter is uniformly bounded from zero, we have $\left|a_{x}\right|\geq \kappa$ for some $\kappa>0$ which depends only on the flow. 
\end{R}
\noindent Let $0<\lambda<1$ be, then there is $t_1>0$ such that $\ds {\lambda^{t_{1}}<\frac{\kappa}{K}\lambda \ \ \text{and} \ \ \lambda^{t_{1}}<\frac{\lambda}{K^{3}}}$, take $n$, such that $t_{n}(x):=\sum^{n}_{i=1}{t_{i}(x)}>t_{1}$ for all $x \in  \Lambda\cap\Xi$, where $t_{i}(x)=t_{+}({\cR}^{i-1}(x))$.\\
\ \\
\noindent So, we have the following proposition:
\begin{Pro}\label{P1GCS}
Let ${\cR}\colon \Xi \to \Xi$ be a Poincar\'e map and $n$ as before. Then $D{\cR}^{n}_{x}(E^{s}_{\Sigma}(x))=E^{s}_{\Sigma^{\prime}}({\cR}^{n}(x))$ at every $x \in \Sigma\in \{\Sigma_i\}_{i}$ and $D{\cR}^{n}_{x}(E^{u}_{\Sigma}(x))=E^{u}_{\Sigma^{\prime}}({\cR}^{n}(x))$ at every $x\in\Lambda\cap\Sigma$ where ${\cR}^{n}(x)\in \Sigma^{\prime}\in \{\Sigma_{i}\}_i$.\\
Moreover, we have that  
\begin{center}
$\left\|D{\cR}^{n}|_{E^{s}_{\Sigma}(x)}\right\| < \lambda$ and $\left\|D{\cR}^{n}|_{E^{u}_{\Sigma}(x)}\right\|>\frac{1}{\lambda}$
\end{center}
at every $x\in\Sigma\in \Xi$.
\end{Pro}
\begin{proof}
The differential of the map ${\cR}^{n}$ at any point $x\in\Sigma$ is given by 
$$D{\cR}^{n}(x)=P_{{\cR}^{n}(x)}\circ D\phi^{t_{n}(x)}|_{T_{x}\Sigma},$$
where $P_{{\cR}^{n}(x)}$ is the projection onto $T_{{\cR}^{n}(x)}\Sigma'$ along the direction of $\phi({\cR}^{n}(x))$. 
\ \\
\ \\
Note that $E^{s}_{\Sigma}$ is tangent to $\Sigma\cap W^{cs}$. Since the center stable manifold $W^{cs}(x)$ is invariant, we have that the stable bundle is invariant:
\begin{center}
$D{\cR}^{n}(x)(E^{s}_{\Sigma}(x))=E^{s}_{\Sigma^{\prime}}({\cR}^{n}(x))$.
\end{center}
Moreover, for all $x\in \Sigma$ 
we have 
\begin{center}
$D\phi^{t_{n}(x)}(E^{u}_{\Sigma}(x))\subset D\phi^{t_{n}(x)}(E^{cu}_{x})=E^{cu}_{{\cR}^{n}(x)}$,
\end{center}
since $P_{{\cR}^{n}(x)}$ is the projection along the vector field, it sends $E^{cu}_{{\cR}^{n}(x)}$ to $E^{u}_{\Sigma^{\prime}}({\cR}^{n}(x))$.\\
This proves that the unstable bundle is invariant restricted to $\Lambda$, that is, $D{\cR}^{n}(x)(E^{u}_{\Sigma}(x))=E^{u}_{\Sigma^{\prime}}({\cR}^{n}(x))$, because they have the same dimension 1.

\indent Next, we prove the expansion and contraction statements. We start by noting that $\left\|P_{{\cR}^{n}(x)}\right\|\leq K$, with $K\geq 1$. Then we consider the basis $\left\{\frac{\phi(x)}{\left\|\phi(x)\right\|},e^{u}_{x}\right\}$ of $E^{cu}_{x}$, where $e^{u}_{x}$ is a unit vector in the direction of $E^{u}_{\Sigma}(x)$ and $\phi(x)$ is the direction of flow. Since the  direction of the flow is invariant by $D\phi^t$, then the matrix of $D\phi^{t}|_{E^{cu}_{x}}$ relative to this basis is upper triangular:
\begin{center}
$D\phi^{t_{n}(x)}|_{E^{cu}_{x}}=\left[\begin{array}{cc} \frac{\left\|\phi({\cR}^{n}(x))\right\|}{\left\|\phi(x)\right\|} & * \\ &  \\ 0 & a \end{array}\right]\quad$
\end{center}
since $D\phi^{t_{n}(x)}(\phi(x))=\phi(\phi^{t_{n}(x)}(x))=\phi({\cR}^{n}(x))$.\\
Then, 
\begin{eqnarray*}
 \left\|D{\cR}^{n}(x)e^{u}_{x}\right\|&=&\left\|P_{{\cR}^{n}(x)}(D\phi^{t_{n}(x)}(x))e^{u}_{x}\right\| =  \left\|ae^{u}_{{\cR}^{n}(x)}\right\| = \left|a\right| \\ &\geq&  \frac{1}{K}\frac{\left\|\phi(x)\right\|}{\left\|\phi({\cR}^{n}(x))\right\|}\left|det(D\phi^{t_{n}(x)}|_{E^{cu}_{x}}\right| \geq  \frac{1}{K^{3}}\lambda^{-t_{n}(x)}\geq K^{-3}\lambda^{-t_{1}}>\frac{1}{\lambda}.
\end{eqnarray*}

\noindent To prove that $\left\|D{\cR}^{n}|_{E^{s}_{\Sigma}(x)}\right\|<\lambda$, let us consider unit vectors, $e^{s}_{x}\in E^{s}_{x}$ and $\hat{e}^{s}_{x}\in E^{s}_{\Sigma}(x)$, and write as in (\ref{E2GCS})
\begin{eqnarray*}
e^{s}_{x}=a_{x}\hat{e}^{s}_{x}+b_{x}\frac{\phi(x)}{\left\|\phi(x)\right\|},
\end{eqnarray*}
with  $\left|a_{x}\right|\geq \kappa$ for some $\kappa>0$ which depends only on the flow. \\
\ \\
Then, since $\ds P_{{\cR}^{n}(x)}\left(\frac{\phi({\cR}^{n}(x))}{\left\|\phi(x)\right\|}\right)=0$ we have that
\begin{eqnarray}\label{E3GCS}
\ds\left\|D{\cR}^{n}(x)\hat{e}^{s}_{x}\right\|&=&\left\|P_{{\cR}^{n}(x)}(D\phi^{t_{n}(x)}(x))\hat{e}^{s}_{x}\right\| \nonumber \\
&=&\left\|P_{{\cR}^{n}(x)}(D\phi^{t_{n}(x)}(x))\left[\frac{1}{a_{x}}\left[{e^{ss}_{x}-b_{x}\frac{\phi(x)}{\left\|\phi(x)\right\|}}\right]\right]\right\| \nonumber \\
&=& \frac{1}{\left|a_{x}\right|}\left\|P_{{\cR}^{n}(x)}(D\phi^{t_{n}(x)}(x))\left[e^{s}_{x}-b_{x}\frac{\phi(x)}{\left\|\phi(x)\right\|}\right]\right\|
\nonumber \\
&=&\frac{1}{\left|a_{x}\right|}\left\|P_{{\cR}^{n}(x)}(D\phi^{t_{n}(x)}(x))(e^{s}_{x})-b_{x}P_{{\cR}^{n}(x)}\left(\frac{\phi(R^{n}(x))}{\left\|\phi(x)\right\|}\right)\right\| \nonumber \\
&\leq & \frac{K}{\kappa}\left\|D\phi^{t_{n}(x)}(x)(e^{ss}_{x})\right\|\leq  \frac{K}{\kappa}\lambda^{t_{n}(x)}\leq\frac{K}{\kappa}\lambda^{t_1}<\lambda \ .   
\end{eqnarray}
\end{proof}

The next step is to prove that there exists $n$ such that $\mathcal{R}^n$ is defined for every point of $\Lambda\cap \Xi$ and consequently, by Proposition \ref{P1GCS}, it is a hyperbolic set for $\mathcal{R}^n$. Moreover, it should be a hyperbolic set for $\mathcal{R}$, since $\Lambda\cap \Xi$ is invariant by $\mathcal{R}$. 

\ \\
For every $x\in \Sigma\ \in \Xi$, we define $W^s(x,\Sigma )$ to be the connected component of $W^{cs}(x)\cap \Sigma$ that contains $x$.
Given $\Sigma\ ,\Sigma^{\prime}\in\Xi$ we set $\Sigma(\Sigma^{\prime})_{n}=\left\{x\in\Sigma:{\cR}^{n}(x)\in \Sigma^{\prime}\right\}$ the domain of the map ${\cR}^{n}$ from $\Sigma$ to $\Sigma^{\prime}$. Remembering relation (\ref{E3GCS}), the tangent direction to each $W^{s}(x,\Sigma)$ is contracted at an exponential rate $\left\|D{\cR}^{n}(x)\hat{e}^{s}_{x}\right\|\leq Ce^{-\beta t_{n}(x)}$, with $C=\frac{K}{\kappa}$ and $\beta=-\log\lambda>0$. Since the cross-section of $\Xi$ are GCS and satisfies (\ref{delta-GCS}) for some  $\delta>0$, then we can take $n$ such that $t_{n}(x)>t_{1}$ as in Proposition \ref{P1GCS} with $t_{1}$ satisfying

\begin{equation}\label{EIP}
Ce^{-\beta t_{1}} \sup\left\{l(W^{s}(x,\Sigma)):x\in \Sigma\right\}<\delta \ \text{and}\ \ Ce^{-\beta t_{1}}<\frac{1}{2},
\end{equation}
where $l(W^{s}(x,\Sigma))$ is the length of $W^{s}(x,\Sigma)$.
Under these conditions we have:
\begin{Le}\label{L11GCS}
Let $n$ be satisfying conditions from Proposition \emph{\ref{P1GCS}}. If ${\cR}^{n}\colon \Sigma(\Sigma^{\prime})_{n} \to \Sigma^{\prime}$ defined by ${\cR}^{n}(x)=\phi^{t_{n}(x)}(x)$. Then, 
\begin{enumerate}
	\item[$(1)$] ${\cR}^{n}(W^{s}(x,\Sigma))\subset W^{s}({\cR}^{n}(x),\Sigma^{\prime})$ for every $x\in \Sigma(\Sigma^{\prime})_{n}$, 
	\item[$(2)$]$d({\cR}^{n}(y),{\cR}^{n}(z))\leq \frac{1}{2}d(y,z)$ for every $y,z\in W^{s}(x,\Sigma)$ and $x\in \Sigma(\Sigma^{\prime})_{n}$.
\end{enumerate}
\end{Le}
\ \\
\indent We let $\left\{U_{\Sigma_{i}}:i=1,\dots,m\right\}$ be a finite cover of $\Lambda$, as in the Lemma \ref{L7} where the $\Sigma_{i}$ is a GCS for each $i$, and we set $T_{3}$ to be an upper bound for the time it takes any point $z\in U_{\Sigma_{i}}$ to leave this tubular neighborhood under the flow, for any $i=1,\dots,l$. We assume, without loss of generality, that $t_{1}$ in Proposition \ref{P1GCS} and (\ref{EIP}) is bigger than $T_{3}$ and  we consider $n$ of Lemma \ref{L11GCS}. If the point $z$ never returns to one of the cross-sections, then the map ${\cR}$ is not defined at $z$. Moreover, by the Lemma \ref{L11GCS}, if ${\cR}^{n}$ is defined for $x\in \Sigma$ for some $\Sigma \in \Xi$, then ${\cR}^{n}$ is defined for every point in $W^{s}(x,\Sigma)$. Hence, the domain of ${\cR}^{n}|\Sigma$ consists of strips of $\Sigma$. The smoothness of $(t,x)\longrightarrow \phi^{t}(x)$ ensure that the strips 
\begin{center}
$\Sigma(\Sigma^{\prime})_{n}=\left\{x\in\Sigma:{\cR}^{n}(x)\in \Sigma^{\prime}\right\}$
\end{center}
have non-empty interior in $\Sigma$ for every $\Sigma , \Sigma^{\prime} \in \Xi$. Note that by the Tubular Flow Theorem and the smoothness of the flow, the map ${\cR}$ is locally smooth for all points $x\in \text{int}\, \Sigma$ such that ${\cR}(x)\in \text{int}\, \Xi$, where $\ds \text{int}\,\Xi=\{\text{int}\,\Sigma_{i}\}^{m}_{i=1}$.
We will denote $\ds \partial^{j}\Xi= \{\partial^{j}\Sigma_{i}\}^{l}_{i=1}$ for $j=s,u$.
\begin{Le}\label{L12GCS}
The set of discontinuities of ${\cR}$ in $\Xi \setminus (\partial^{s}\Xi\cup \partial^{u}\Xi)$ is contained in the set of point $x\in \Xi \setminus(\partial^{s}\Xi\cup \partial^{u}\Xi)$ such that, ${\cR}(x)$ is defined and belongs to $(\partial^{s}\Xi\cup \partial^{u}\Xi)$.
\end{Le}
\begin{proof}
Let $x$ be a point in $\Sigma \setminus (\partial^{s}\Sigma\cup \partial^{u}\Sigma)$ for some $\Sigma\in \Xi$, not satisfying the condition. Then ${\cR}(x)$ is defined and ${\cR}(x)$ belongs to the interior of some cross-section $\Sigma^{\prime}$. By the smoothness of the flow, we have that ${\cR}$ is smooth in a neighborhood of $x$ in $\Sigma$. Hence, any discontinuity point for ${\cR}$ must be in the condition of the Lemma.
\end{proof}
\noindent Let $D_{j}\subset \Sigma_{j}$ be the set of points sent by ${\cR}^{n}$ into stable boundary points of some Good Cross-Section of $\Xi$, if we define the set  
$$L_{j}=\left\{W^{s}(x,\Sigma_{j}):x\in D_{j}\right\},$$
then the  Lemma \ref{L11GCS} implies that $L_{j}=D_{j}$. Let $B_{j}\subset \Sigma_{j}$ be the set of points sent by ${\cR}^{n}$ into unstable boundary points of some Good Cross-Section of $\Xi$. Denote $$\Gamma_{j}=\bigcup_{x\in D_{j}}W^{s}(x,\Sigma_{j})\cup B_{j} \ \ \text{and}\ \  \Gamma=\bigcup \Gamma_{j}\cup (\partial^{s}\Xi\cup \partial^{u}\Xi).$$\\
Then, ${\cR}^{n}$ is smooth in the complement $\Xi\setminus \Gamma$ of $\Gamma$. Observe that if $x\in D_{j}$ for some $j\in\left\{1,\dots,l\right\}$, then 
$${\cR}^{n}(W^{s}(x,\Sigma_{j}))\subset \partial^{s}\Sigma^{\prime} \ \ \text{for\ some} \ \ \Sigma^{\prime}\in \Xi.$$
We know that $\partial^{s}\Xi \cap \Lambda=\emptyset$, then ${\cR}^{n}(W^{s}(x,\Sigma_{j}))\cap \Lambda=\emptyset$ for all $x \in D_j$, which implies that $W^{s}(x,\Sigma_{j})\cap \Lambda=\emptyset$ for all $x\in D_{j}$. However, if $x\in B_{j}$, then ${\cR}^{n}(x)\in \partial^{u}\Sigma^{\prime}$ for some $\Sigma^{\prime}\in \Xi$ and again we know that $\partial^{u}\Xi \cap \Lambda=\emptyset$, this implies that $B_{j}\cap \Lambda=\emptyset$. Therefore, $\Gamma_{j}\cap \Lambda=\emptyset$ for all $j\in\left\{1,\dots,l\right\}$, so $\Gamma\cap \Lambda=\emptyset$. The latter arguments proved the following: 
\begin{Le}\label{L14}
If $x\in\Lambda\cap\Xi$, then ${\cR}^{n}(x)$ is defined and ${\cR}^{n}(x)\in \emph{int}\,\Xi$.\\
\end{Le}
\begin{proof}[\emph{\textbf{Proof of Lemma \ref{LHPM}}}]
Note simply that by Lemma \ref{L14} the set $\Lambda\cap \Xi$ is an invariant set for ${\cR}^{n}$ and by Proposition \ref{P1GCS}, $\Lambda\cap \Xi$ is hyperbolic set for ${\cR}^{n}$ and since $\Lambda\cap \Xi$ is invariant for ${\cR}$, then $\Lambda\cap \Xi$ is hyperbolic for ${\cR}$, and 
$$\Lambda\cap \Xi\subset  \bigcap_{n\in \mathbb{Z}}{\cR}^{-n}(\Xi)=\Delta.$$
\end{proof}

\subsection{Regular Cantor Sets}\label{RCS}
Let $\mathbb{A}$ be a finite alphabet, $\mathbb{B}$ a subset of $\mathbb{A}^{2}$, and $\Sigma_{\mathbb{B}}$ the subshift of finite type of $\mathbb{A}^{\mathbb{Z}}$ with allowed transitions $\mathbb{B}$. We will always assume that $\Sigma_{\mathbb{B}}$ is topologically mixing and that every letter in $\A$ occurs in $\Sigma_{\mathbb{B}}$.

\indent An {\it expansive map of type\/} $\Sigma_{\mathbb{B}}$ is a map $g$ with the following properties:
\begin{itemize}
\item[(i)] the domain of $g$ is a disjoint union
$\ds\bigcup_{\mathbb{B}}I(a,b)$. Where for each $(a,b)$,\,\, $I(a,b)$ is a compact subinterval of $I(a) := [0,1]\times\{a\}$;
\item[(ii)] for each $(a,b) \in \mathbb{B}$, the restriction of $g$ to $I(a,b)$ is a smooth diffeomorphism onto $I(b)$
satisfying $|Dg(t)| > 1$ for all $t$.
\end{itemize} 
\noindent The {\it regular Cantor set\/} associated to $g$ is the maximal invariant set
$$K = \bigcap_{n\ge0} g^{-n}\bigg(\bigcup_{\B} I(a,b)\bigg).$$
\noindent Let $\Sigma^+_{\mathbb{B}}$ be the unilateral subshift associated to $\Sigma_{\mathbb{B}}$. There exists a unique homeomorphism $h\colon \Sigma^{+}_{\B} \to K$ such that
$$
h(\underline{a}) \in I(a_0), \text{ for } \underline{a} = (a_0,a_1,\dots) \in \Sigma^+_{\mathbb{B}}
\ \ and \ \ 
h\circ\sigma =g \circ h,$$
where $\sigma^{+}\colon \Sigma_{\B}^{+} \to \Sigma_{\B}^{+}$, is defined as follows $\sigma^{+}((a_{n})_{n\geq 0})=(a_{n+1})_{n\geq0}$.

\subsection {Expanding Maps Associated to a Horseshoe}\label{sec EMAH}
\noindent Let $\Lambda$ be a horseshoe associated to $C^{2}$-diffeomorphism $\varphi$ on a surface $M$ and consider a finite collection $(R_{a})_{a\in\mathbb{A}}$ of disjoint rectangles of $M$, which are a Markov partition of $\Lambda$. Define the sets 
 $$W^{s}(\Lambda,R)=\bigcap_{n\geq0}\varphi^{-n}(\bigcup_{a\in \mathbb{A}}R_{a}),$$
$$W^{u}(\Lambda,R)=\bigcap_{n\leq0}\varphi^{-n}(\bigcup_{a\in \mathbb{A}}R_{a}).$$
There is a $r>1$ and a collection of $C^{r}$-submersions $(\pi_{a}:R_{a}\rightarrow I(a))_{a\in\mathbb{A}}$, satisfying the following property:\\
\ \\
If $z,z^{\prime}\in R_{a_{0}}\cap \varphi^{-1}(R_{a_{1}})$ and $\pi_{a_{0}}(z)=\pi_{a_{0}}(z^{\prime})$, then we have $$\pi_{a_{1}}(\varphi(z))=\pi_{a_{1}}(\varphi(z^{\prime})).$$

\noindent In particular, the connected components of $W^{s}(\Lambda,R)\cap R_{a}$ are the level lines of $\pi_{a}$. Then we define a mapping $g^{u}$ of class $C^{r}$ (expansive of type $\Sigma_{\mathbb{B}}$) by the formula
$$g^{u}(\pi_{a_{0}}(z))=\pi_{a_{1}}(\varphi(z))$$
\noindent for $(a_{0},a_{1})\in \B$, $z\in R_{a_{0}}\cap\varphi^{-1}(R_{a_{1}})$.
The regular Cantor set $K^{u}$ defined by $g^{u}$, describes the geometry transverse of the stable foliation $W^{s}(\Lambda,R)$.
Analogously, we can describe the geometry  transverse of the unstable foliation $W^{u}(\Lambda,R)$ using a regular Cantor set $K^{s}$ define by a mapping $g^{s}$ of class $C^{r}$ (expansive of type $\Sigma_{\B}$).\\
\noindent Also, the horseshoe $\Lambda$ is locally the product of two regular Cantor sets $K^{s}$ and $K^{u}$. So, the Hausdorff dimension of $\Lambda$, $HD(\Lambda)$ is equal to
$HD(K^{s}\times K^{u})$, but for regular Cantor sets, we have that $HD(K^{s}\times K^{u})=HD(K^{s})+HD(K^{u})$. Thus $HD(\Lambda)=HD(K^{s})+HD(K^{u})$ (cf. \cite[chap 4]{PT}).

\subsection{Intersections of Regular Cantor Sets and Property $V$}\label{SIRCS}
\noindent Let $r$ be a real number $> 1$, or $r=+\infty$. The space of $C^r$ expansive maps of type $\Sigma$ (cf. Subsection \ref{RCS}), endowed with the $C^r$ topology, will be denoted by $\Omega_\Sigma^r$\,. The union $\Omega_\Sigma = \ds\bigcup_{r>1} \Omega_\Sigma^r$ is endowed with the inductive limit topology.\\

\noindent  Let $\Sigma^- = \{(\theta_n)_{n\leq 0}\,, (\theta_i,\theta_{i+1})
\in \B \text{ for } i < 0\}$. We equip $\Sigma^-$ with the following
ultrametric distance: for $\und{\theta} \ne \und{\widetilde\theta} \in \Sigma^-$, set

\[ d(\und{\theta},\und{\widetilde\theta}) = \left\{ \begin{array}{lll}
         \ \ \ 1 & \mbox{if \ $\theta_0 \ne \widetilde{\theta}_0$};\\
           & \\
       |I(\und{\theta} \wedge \und{\widetilde\theta})|& \mbox{otherwise}\end{array} \right., \] 

\noindent where $\und{\theta} \wedge \und{\widetilde\theta} = (\theta_{-n},\dots,\theta_0)$ if
$\widetilde\theta_{-j} = \theta_{-j}$ for $0 \le j \le n$ and $\widetilde\theta_{-n-1}
\ne \theta_{-n-1}$\,.
\\
\indent Now, let $\und{\theta} \in \Sigma^-$; for $n > 0$, let $\und{\theta}^n = (\theta_{-n},\dots,\theta_0)$, and let $B(\und{\theta}^n)$ be the affine map from
$I(\und{\theta}^n)$ onto $I(\theta_0)$ such that the diffeomorphism $k_n^{\und{\theta}}
= B(\und{\theta}^n) \circ f_{\und{\theta}^n}$ is orientation preserving.\\ 
\noindent We have the following well-known result (cf. \cite{Su}):
\\ 

\noindent{\bf Proposition}. \textit{Let} $r \in (1,+\infty)$, $g \in \Omega_\Sigma^r$.
\begin{enumerate}

	\item \textit{For any $\und{\theta} \in \Sigma^-$, there is a diffeomorphism $k^{\und{\theta}} \in \text{Diff}_+^{\ r}(I(\theta_0))$ such that $k_n^{\und{\theta}}$ converge to $k^{\und{\theta}}$ in $\text{Diff}_{+}^{\ r'}(I(\theta_0))$, for any $r'< r$, uniformly in $\und{\theta}$. The convergence is also uniform in a neighborhood of $g$ in $\Omega_\Sigma^r$\,.}

\item  \textit{If $r$ is an integer or $r = +\infty$,\,\,\,$k_n^{\und{\theta}}$ converge to $k^{\und{\theta}}$ in $\text{Diff}_+^r(I(\theta_0))$. More precisely, for every $0 \leq j \leq r-1$, there is a constant $C_j$ (independent on $\und{\theta}$) such that
$$
\left| D^j \, \log \, D \left[k_n^{\und{\theta}} \circ (k^{\und{\theta}})^{-1}\right](x)\right| \leq C_j|I(\und{\theta}^n)|.
$$
It follows that $\und{\theta} \to k^{\und{\theta}}$ is Lipschitz in the following sense: for $\theta_0 = \widetilde\theta_0$, we have
$$
\left|D^j \, \log \, D\big[k^{\und{\widetilde\theta}} \circ (k^{\und{\theta}})^{-1}\big](x)\right| \leq C_j\,d(\und{\theta}, \und{\widetilde\theta}).
$$
}
\end{enumerate}

\ \\
\noindent Let $r \in (1,+\infty]$. For $a \in \A$, we denote by ${\cal{P}}^{r}(a)$ the space of $C^r$-embeddings of $I(a)$ into $\re$, endowed with the $C^r$ topology. The affine group $Aff(\re)$ acts by composition on the left on ${\cal{P}}^r(a)$, the quotient space being denoted by $\overline{\cal{P}}^r(a)$. We also consider ${\cal{P}}(a) = \ds\bigcup_{r>1} {\cal{P}}^r(a)$ and $\overline{\cal{P}}(a) = \ds\bigcup_{r>1} \overline{\cal{P}}^r(a)$, endowed with the inductive limit topologies.

\begin{R}\label{r=1}
In $\cite{MY}$ is considered ${\cal{P}}^{r}(a)$ for $r \in (1,+\infty]$, but all the definitions and results involving ${\cal{P}}^{r}(a)$ can be obtained considering $r\in [1,+\infty]$.
\end{R}

\noindent Let $\mathcal{A} =(\und{\theta}, A)$, where $\und{\theta} \in \Sigma^-$ and $A$ is now an {\it affine\/} embedding of $I(\theta_0)$ into $\re$. We have a canonical map
\begin{eqnarray*}
\cal{A} & \to & {\cal{P}}^r = \bigcup_{\A} {\cal{P}}^r(a)\\
(\und{\theta},A) &\mapsto & A\circ k^{\und{\theta}} \ \ (\in {\cal{P}}^r(\theta_0)).
\end{eqnarray*}

\noindent Now assume we are given two sets of data $(\A,\B,\Sigma,g)$, $({\A}',{\B}',\Sigma',g')$ defining regular Cantor sets $K$, $K'$.\\
We define as in the previous the spaces $\mathcal{P} = \ds\bigcup_{\mathbb{A}}{\mathcal{P}}(a)$ and ${\cal{P}}'= \ds\bigcup_{{\A}'} {\cal{P}}(a')$.

\noindent A pair $(h,h')$, 
$(h \in {\cal{P}}(a), h'\in {\cal{P}} '(a'))$ is called a {\it smooth configuration\/} for $K(a)=K\cap I(a)$, $K'(a')=K'\cap I(a')$. Actually, rather than working in the product $\cal{P} \times {\cal{P}}'$, it is better to go to the quotient $Q$ by the diagonal action of the affine group $Aff(\re)$. Elements of $Q$ are called {\it smooth relative configurations\/} for $K(a)$, $K'(a')$.

\noindent We say that a smooth configuration $(h,h') \in {\cal{P}}(a)\times {\cal{P}}(a')$ is
\begin{itemize}
\item {\it linked\/} if $h(I(a)) \cap h'(I(a')) \ne \emptyset$;
\item {\it intersecting\/} if $h(K(\und{a})) \cap h'(K(\und{a}')) \ne \emptyset$, where $K(\und{a})=K\cap I(\und{a})$ and $K(\und{a}')=K\cap I(\und{a}')$;
\item {\it stably intersecting\/} if it is still intersecting when we perturb it in $\cal{P}\times\cal{P}'$, and we perturb $(g,g')$ in $\Omega_\Sigma \times \Omega_{\Sigma'}$\,.
\end{itemize}

\noindent All these definitions are invariant under the action of the affine group and, therefore, make sense for smooth relative configurations.\\

\noindent As in previous, we can introduce the spaces $\cal{A}$, ${\cal{A}}'$ associated to the limit geometries of $g$,\,\,$g'$, respectively. We denote by $\cal{C}$ the quotient of $\cal{A}\times{\cal{A}}'$ by the diagonal action on the left of the affine group. An element of $\cal{C}$, represented by $(\und\theta,A) \in \cal{A}$,\,\, $(\und{\theta}', A') \in {\cal{A}}'$, is called a relative configuration of the limit geometries determined by $\und{\theta}$, $\und{\theta}'$. We have canonical maps
\begin{eqnarray*}
\cal{A}\times{\cal{A}}'&\to & \cal{P}\times{\cal{P}}'\\
\cal{C} &\to & Q
\end{eqnarray*}
 allowing  to define linked, intersecting, and stably intersecting configurations at the level of $\cal{A}\times{\cal{A}}'$ or $\cal{C}$.\\

\noindent{\bf Remark}: For a configuration $((\und{\theta}, A), (\und{\theta}',A'))$ of limit geometries, one could also consider the {\it weaker\/} notion of stable intersection obtained by considering perturbations of $g$, $g'$ in $\Omega_\Sigma \times \Omega_{\Sigma'}$ and perturbations of $(\und{\theta},A)$, $(\und{\theta}',A')$ in $\cal{A}\times{\cal{A}}'$. We do not know of any example of expansive maps $g$, $g'$, and configurations $(\und{\theta},A)$, $(\und{\theta}',A')$ which are stably intersecting in the weaker sense, but not in the stronger sense.\\

\noindent We consider the following subset $V$ of $\Omega_\Sigma \times \Omega_{\Sigma'}$\,. A pair $(g, g')$ belongs to $V$ if for any $[(\und{\theta},A), (\und{\theta}',A')] \in \cal{A} \times {\cal{A}}'$ there is a translation $R_t$ (in $\re$) such that $(R_t\circ A \circ k^{\und{\theta}}, A'\circ k^{\prime\und{\theta}'})$ is a stably intersecting configuration.\\
\begin{Defi}\label{Property V}
We say that a pair $(\psi, \Lambda)$, where $\Lambda$ is a horseshoe for $\psi$, has the property $V$ if the stable and unstable cantor sets have the property $V$ in the above sense.
\end{Defi}
The more important result in this setting is:
\begin{T}[Moreira-Yoccoz \cite{MY1}]\label{MY1}Let $\varphi$ be a $C^{\infty}$ diffeomorphism with a horseshoe $\Lambda$. Let $K^s$,  $K^u$ are the stable and unstable  Cantor sets  respectively. Suppose that $HD(K^s)+HD(K^u)>1$.   If\,  $\mathcal{U}$ is  sufficiently small neighborhood $\varphi$ in $Diff^{\infty}(M)$, there is an open and dense set ${\mathcal{U}}^{\ast}\subset\mathcal{U}$
such that, for every $\psi\in {\mathcal{U}}^{\ast}$ the pair $(\psi,\Lambda_{\psi})$ has the property $V$.
\end{T}

\subsection{The Birkhoff Invariant}\label{BI}
Let $f\colon (\re^2,0)\to (\re^2,0)$ be a germ of diffeomorphism area-preserving (in dimension two is symplectic) and $0$ a hyperbolic fixed point with eigenvalues $\lambda$ and $\lambda^{-1}$, then the Birkhoff normal form (cf. \cite{Moser}) says that there is an area-preserving change of coordinates $\Phi$ such that $\Phi^{-1}\circ f \circ \Phi=N$, where 
$N(x,y)=(U(xy)x,U^{-1}(xy)y)$ and $U(xy)$ is a power series $\lambda+U_2xy+\cdots$ convergent in a neighborhood of $x=y=0$. In other words, in this coordinates $f$ can be written by 
\begin{equation}\label{EBI}
f(x,y)=(\lambda x(1+axy+\mathcal{O}(\Vert(x,y)\Vert^{4})),\lambda^{-1} y(1-axy+\mathcal{O}(\Vert(x,y)\Vert^{4})))
\end{equation}
and the number $a$ is called the \textit{Birkhoff Invariant} of $f$.
\begin{Le}\label{Birkhoff Invariant}
The Birkhoff invariant for diffeomorphism area-preserving in $(\re^2,0)$ only depends on\, $3$-jets in $0$, $J^{3}(0)$. Moreover, the set of diffeomorphism area-preserving in $(\re^2,0)$ such that the Birkhoff invariant is non-zero is open, dense, and invariant in $J^{3}(0)$.
\end{Le}
\begin{proof}
For the proof of  \cite[Theorem 1 and 2]{Moser}, we have the first part and opening. For density, suppose that for some  $f\colon (\re^2,0)\to (\re^2,0)$, the Birkhoff invariant is zero, then for $\epsilon>0$ we consider the function 
$N_\epsilon(x,y):=(\lambda x(1+\mathcal{O}(\Vert(x,y)\Vert^{4})),\lambda^{-1} y(1+\mathcal{O}(\Vert(x,y)\Vert^{4})))+\epsilon(x^2y,-xy^2)$, then the function $f_{\epsilon}=\Phi\circ N_\epsilon \circ \Phi^{-1}$ is area-preserving diffeomorphism close to $f$ with the Birkhoff invariant  $\epsilon$. \\
Let $f$, $g$ be as above and suppose that the Birkhoff invariant for $f$ is non-zero, then $g^{-1}\circ f\circ g$ has the Birkhoff invariant non zero. Indeed, by the Birkhoff Normal Form \cite[Theorem 1]{Moser},  there is an area-preserving change of coordinates $\Phi$ such that $\Phi^{-1}\circ g^{-1}\circ f\circ g \circ \Phi$ has the form (\ref{EBI}), then $(g\circ \Phi)^{-1}\circ f \circ (g\circ \Phi)$ has the form (\ref{EBI}). In other words, there is another area-preserving change of coordinates 
$g\circ \Phi$ such that $f$ has the form (\ref{EBI}), but by the unicity of the Birkhoff normal form (see \cite[page 674]{Moser}), we have that the Birkhoff invariant of $g^{-1}\circ f\circ g$ is equal to the Birkhoff invariant of $f$, therefore non-zero.
\end{proof}

$$\bf{Acknowledgments}$$
The author would like to thank Carlos Gustavo Moreira (Gugu) and Carlos Matheus for very helpful discussions and suggestions during the preparation of this paper. \\
 \ \\
\noindent \textbf{Sergio Augusto Roma\~na Ibarra}\\
Universidade Federal do Rio de Janeiro\\
Av. Athos da Silveira Ramos 149, Centro de Tecnologia \ - Bloco C \ - Cidade Universit\'aria \ - Ilha do Fund\~ao, cep 21941-909 \\
Rio de Janeiro-Brasil\\
E-mail: sergiori@im.ufrj.br \\
\ \\
\bibliographystyle{alpha}	
\bibliography{bibtex}

\end{document}